\documentclass[11pt]{amsart}
\usepackage{amsthm}
\usepackage{amsmath,amstext,amsthm,amsfonts,amssymb,amsthm}
\usepackage{mathrsfs}
\usepackage{multicol}
\usepackage{latexsym}
\usepackage{color}
\usepackage{graphicx}
\usepackage{epsf}
\usepackage{epsfig}
\usepackage{epic}
\usepackage{graphics}
\usepackage{verbatim}
\usepackage{color}
\usepackage{hyperref}
\usepackage{bm}
\newtheorem{theorem}{Theorem}[section]
\newtheorem{lemma}[theorem]{Lemma}
\newtheorem{corollary}[theorem]{Corollary}
\theoremstyle{definition}
\newtheorem{definition}[theorem]{Definition}
\newtheorem{proposition}[theorem]{Proposition}
\newtheorem{example}[theorem]{Example}

\theoremstyle{remark}
\newtheorem{remark}[theorem]{Remark}
\numberwithin{equation}{section}

\newcommand{\q}{\mathfrak{q}}
\newcommand{\HI}{\mathfrak{H}}

\newcommand{\C}{\mathbb{C}}

\newcommand{\B}{\mathcal{B}}

\newcommand{\qu}{\mathfrak{q}}
\newcommand{\pu}{\mathfrak{p}}
\newcommand{\oqu}{\overline{\mathfrak{q}}}
\newcommand{\IV}{I_{V_H^R}}

\newcommand{\quat}{\mathbb H}
\newcommand{\R}{\Bbb R}
\newcommand{\Z}{\mathbb Z}
\newcommand{\mc}{\mathcal}

\newcommand{\be}{\begin{equation}}
\newcommand{\en}{\end{equation}}
\newcommand{\D}{{\mc D}}

\newcommand{\N}{\mathbb N}

\newcommand{\bedefin}{\begin{defi}}
	\newcommand{\findefi}{\end{defi} \medskip}
\newcommand{\betheo}{\begin{theorem}$\!\!${\bf \,\,\,}}
	\newcommand{\entheo}{\end{theorem}}
\newcommand{\enth}{\end{theorem}}
\newcommand{\becor}{\begin{cor}$\!\!${\bf .}}
	\newcommand{\encor}{\end{cor}}
\newcommand{\belem}{\begin{lem}$\!\!${\bf .}}
	\newcommand{\enlem}{\end{lem}}
\newcommand{\bea}{\begin{eqnarray}}
\newcommand{\ena}{\end{eqnarray}}
\newcommand{\beano}{\begin{eqnarray*}}
	\newcommand{\enano}{\end{eqnarray*}}
\newcommand{\bee}{\begin{enumerate}}
	\newcommand{\ene}{\end{enumerate}}
\newcommand{\bei}{\begin{itemize}}
	\newcommand{\eni}{\end{itemize}}
\newcommand{\betab}{\begin{tabular}}
	\newcommand{\entab}{\end{tabular}}

\newcommand{\Iop}{{\mathbb{I}_{V_{\mathbb{H}}^{R}}}}

\newcommand{\bk}{\mathbf k}

\newcommand{\bi}{\mathbf i}
\newcommand{\bj}{\mathbf j}

\newcommand{\apo}{\sigma_{ap}^S}
\newcommand{\vr}{V_\quat^R}
\newcommand{\ur}{U_\quat^R}
\newcommand{\ra}{\text{ran}}
\newcommand{\kr}{\text{ker}}

\newcommand{\sus}{\sigma_{su}^S}
\newcommand{\rka}{\rho_{ka}^S}
\newcommand{\ska}{\sigma_{ka}^S}
\newcommand{\ov}{\overline{V}}
\newcommand{\oa}{\overline{A}}
\newcommand{\op}{\overline{\phi}}
\newcommand{\lr}{\rho_A^S}
\newcommand{\ls}{\sigma_A^S}
\newcommand{\rgka}{\rho_{gk}^S}
\newcommand{\sgka}{\sigma_{gk}^S}
\newcommand{\res}{\rho_{es}^S}
\newcommand{\ses}{\sigma_{es}^S}
\begin{document}
\title[Kato S-spectrum]{Kato S-spectrum in the quaternionic setting}
\author{B. Muraleetharan$^{\dagger}$ and %I. Sabadini$^*$, 
K. Thirulogasanthar$^{\ddagger}$}
\address{$^{\dagger}$ Department of mathematics and Statistics, University of Jaffna, Thirunelveli, Sri Lanka.}
\address{$^{\ddagger}$ Department of Computer Science and Software Engineering, Concordia University, 1455 De Maisonneuve Blvd. West, Montreal, Quebec, H3G 1M8, Canada.}
%\address{$^*$Dipartimento di Matematica, Politecnico di Milano, Via E. Bonardi, 9, 20133, Milano, Italy.}
\email{bbmuraleetharan@jfn.ac.lk and santhar@gmail.com  %irene.sabadini@polimi.it. 
}
\date{\today}
\date{\today}
\begin{abstract}
In a right quaternionic Hilbert space, for a bounded right linear operator, the Kato S-spectrum is introduced and studied to a certain extent. In particular, it is shown that the Kato S-spectrum is a non-empty compact subset of the S-spectrum and it contains the boundary of the S-spectrum.  Using right-slice regular functions, local S-spectrum, at a point of a right quaternionic Hilbert space, and the local spectral subsets are introduced and studied.  The S-surjectivity spectrum and its connections to the Kato S-spectrum, approximate S-point spectrum and local S-spectrum are investigated. The generalized Kato S-spectrum is introduced and it is shown that the generalized Kato S-spectrum is a compact subset of the S-spectrum.
\end{abstract}
\maketitle
\pagestyle{myheadings}
%%%%%%%%%%%%%%%%%%%%%%%%%%%%%%%%%%%%%%%%%%%%%%%%%%%%%%%%%%%%%%%%%%%%%%%%
\section{Introduction}
In complex spectral theory, the spectrum of a bounded linear operator on a Hilbert space or Banach space can be  divided into several subsets depending on the purpose of the investigation. Further, some of these subsets can also be expressed and analyzed  in terms of the local spectrum at a point of the Hilbert space or Banach space. The local spectral theory is closely linked to vector-valued analytic functions. As one of these subsets, the so-called Kato spectrum was first introduced by Apostol for bounded linear operators on a Hilbert space \cite{Apo}, and then investigated by several authors on Banach spaces. The Kato spectrum has close link to surjectivity spectrum and approximate point spectrum under certain assumptions. For a detail account on the complex theory see \cite{Ai, La,Ben}, and the many references therein.\\

In the complex setting, in a complex Hilbert space $\HI$, for a bounded linear operator, $A$, the spectrum is defined as the set of complex numbers $\lambda$ for which the operator  $Q_\lambda(A)=A-\lambda \mathbb{I}_\HI$, where $\mathbb{I}_{\HI}$ is the identity operator on $\HI$,  is not invertible. In the quaternionic setting, let $\vr$ be a separable right quaternionic Hilbert space,  $A$ be a bounded right linear operator, and $R_\qu(A)=A^2-2\text{Re}(\qu)A+|\qu|^2\Iop$, with $\qu\in\quat$, the set of all quaternions, be the pseudo-resolvent operator. The S-spectrum is defined as the set of quaternions $\qu$ for which $R_\qu(A)$ is not invertible. In the complex case various classes of spectra, such as approximate point spectrum, essential spectrum, Weyl spectrum, Browder spectrum, Kato spectrum, surjectivity spectrum etc. are defined by placing restrictions on the operator $Q_\lambda(A)$ \cite {Ai, Kub, La}. In this regard, in the quaternionic setting, in order to define similar classes of spectra it is natural to place the same restrictions to the operator $R_\qu(A)$.\\

Due to the non-commutativity, in the quaternionic case  there are three types of  Hilbert spaces: left, right, and two-sided, depending on how vectors are multiplied by scalars. This fact can entail several problems. For example, when a Hilbert space $\mathcal H$ is one-sided (either left or right) the set of linear operators acting on it does not have a linear structure. Moreover, in a one sided quaternionic Hilbert space, given a linear operator $A$ and a quaternion $\mathfrak{q}\in\quat$, in general we have that $(\mathfrak{q} A)^{\dagger}\not=\overline{\mathfrak{q}} A^{\dagger}$ (see \cite{Ad,Mu} for details). These restrictions can severely prevent the generalization  to the quaternionic case of results valid in the complex setting. Even though most of the linear spaces are one-sided, it is possible to introduce a notion of multiplication on both sides by fixing an arbitrary Hilbert basis of $\mathcal H$.  This fact allows to have a linear structure on the set of linear operators, which is a minimal requirement to develop a full theory \cite{MTS,BT}. However, in this manuscript we develop the theory on $\vr$ without introducing a left multiplication on it. \\

As far as we know, the local $S$-spectral theory, Kato $S$-spectrum and the surjectivity $S$-spectrum  have not been studied in the quaternionic setting yet. In this regard, in this note we investigate these spectra in the quaternionic setting. The surjectivity $S$-spectrum  has close connection with the approximate $S$-point spectrum, the local $S$-spectrum and Kato $S$-spectrum. In the complex case, the local spectrum, at a point in $\HI$, is defined in terms of operator-valued analytic functions \cite{Ai, La}. There have been several attempts to define analyticity in the quaternionic setting by mimicking the complex setting \cite{Am}. However, the most promising, and recent attempt
was the slice-regularity, that is, the slice-regular functions are the quaternionic counterpart of the complex analytic functions \cite{NFC,ghimorper, Gra2, GSS,S}. In this regard, we define the local $S$-spectrum in terms of slice-regular functions.\\

Apart from the non-commutativity of quaternions, due to the structure of the operator $R_\qu(A)$ we have experienced severe difficulties in extending several results valid in the complex setting to quaternions. For example, for $\lambda, \mu\in\C$, $Q_\lambda(A)=Q_\mu(A)-(\lambda-\mu)\mathbb{I}_{\HI}$ and this equality plays an important role in proofs of several local spectral results \cite{Ai, La}. Unfortunately, a similar equality, in a satisfactory way, could not be obtained for the operator $R_\qu(A)$ by us. Even if we restrict $R_\qu(A)$ to a complex slice within quaternions $Q_\lambda(A)\not=R_\lambda(A)$, therefore, we cannot expect all the results valid in the complex setting to hold for quaternions. However, by imposing additional conditions analogous results may be obtained.\\

The article is organized as follows. In section 2 we introduce the set of quaternions, quaternionic Hilbert spaces and their bases, and slice-regularity as needed for the development of this article, which may not be familiar to a broad range of audience. In section 3 we define and investigate, as needed, right linear operators and their properties. In section 3.1 we deal with the S-spectrum and its major partitions. In section 4 we study the surjectivity $S$-spectrum and its connection to approximate $S$-point spectrum and to the $S$-spectrum. We also characterize the $S$-spectrum in terms of the spectral radius and the lower bound of a bounded right linear operator. In section 5 we study hyper-kernel, hyper range, semi-regular operators, algebraic core and analytic core of an operator. The proofs of most of the results in this section follow its complex counterpart. In this respect we give references for complex proofs. In section 6 we study local $S$-spectrum, local $S$-spectral subspaces and the single-valued extension property (SVEP). In particular, we show that when a quaternionic right linear operator $A$ has SVEP then the $S$-surjectivity spectrum coincides with the $S$-spectrum while its adjoint $A^\dagger$ has SVEP then $S$-approximate point spectrum coincides with the $S$-spectrum. In section 7 we introduce and study the Kato $S$-spectrum. In particular, we show that the Kato $S$-spectrum is a compact subset of the $S$-spectrum and it contains the boundary of the $S$-spectrum. We also examine connections between Kato $S$-spectrum and the $S$-surjectivity and $S$-approximate point spectra. It is also shown that if operators $A$ and $A^\dagger$ have SVEP then the Kato $S$-spectrum coincides with the $S$-spectrum. In section 8 we introduce the generalized Kato decomposition, generalized Kato S-spectrum and essentially semi-regular $S$-spectrum. In particular, we show that generalized Kato S-spectrum and essentially semi-regular $S$-spectrum are compact subsets of the $S$-spectrum. Section 9 ends the manuscript with a conclusion.

%%%%%%%%%%%%%%%%%%%%%%%%%%%%%%%%%%%%%%%%%%%%%%%%%%%%%%%%%%%%%%%%%%%%%%
\section{Mathematical preliminaries}
In order to make the paper self-contained, we recall some facts about quaternions which may not be well-known.  For details we refer the reader to \cite{Ad, NFC,ghimorper,Vis}.
\subsection{Quaternions}
Let $\quat$ denote the field of all quaternions and $\quat^*$ the group (under quaternionic multiplication) of all invertible quaternions. A general quaternion can be written as
$$\qu = q_0 + q_1 \bi + q_2 \bj + q_3 \bk, \qquad q_0 , q_1, q_2, q_3 \in \mathbb R, $$
where $\bi,\bj,\bk$ are the three quaternionic imaginary units, satisfying
$\bi^2 = \bj^2 = \bk^2 = -1$ and $\bi\bj = \bk = -\bj\bi,  \; \bj\bk = \bi = -\bk\bj,
\; \bk\bi = \bj = - \bi\bk$. The quaternionic conjugate of $\qu$ is
$$ \overline{\qu} = q_0 - \bi q_1 - \bj q_2 - \bk q_3 , $$
while $\vert \qu \vert=(\qu \overline{\qu})^{1/2} $ denotes the usual norm of the quaternion $\qu$.
If $\qu$ is a non-zero element, it has the inverse
$
\qu^{-1} =  \dfrac {\overline{\qu}}{\vert \qu \vert^2 }.$
Finally, the set
\begin{eqnarray*}
\mathbb{S}&=&\{I=x_1 \bi+x_2\bj+x_3\bk~\vert
~x_1,x_2,x_3\in\mathbb{R},~x_1^2+x_2^2+x_3^2=1\},
\end{eqnarray*}
contains all the elements whose square is $-1$. It is a $2$-dimensional sphere in $\mathbb H$.
%%%%%%%%%%%%%%%%%%%%%%%%%%%%%%%%%%%%%%%%%%%%%%%%%%%%%%%%%
\subsection{Quaternionic Hilbert spaces}
In this subsection we  discuss right quaternionic Hilbert spaces. For more details we refer the reader to \cite{Ad,ghimorper,Vis}.
\subsubsection{Right quaternionic Hilbert Space}
Let $V_{\quat}^{R}$ be a vector space under right multiplication by quaternions.  For $\phi,\psi,\omega\in V_{\quat}^{R}$ and $\qu\in \quat$, the inner product
$$\langle\cdot\mid\cdot\rangle_{V_{\quat}^{R}}:V_{\quat}^{R}\times V_{\quat}^{R}\longrightarrow \quat$$
satisfies the following properties
\begin{enumerate}
	\item[(i)]
	$\overline{\langle \phi\mid \psi\rangle_{V_{\quat}^{R}}}=\langle \psi\mid \phi\rangle_{V_{\quat}^{R}}$
	\item[(ii)]
	$\|\phi\|^{2}_{V_{\quat}^{R}}=\langle \phi\mid \phi\rangle_{V_{\quat}^{R}}>0$ unless $\phi=0$, a real norm
	\item[(iii)]
	$\langle \phi\mid \psi+\omega\rangle_{V_{\quat}^{R}}=\langle \phi\mid \psi\rangle_{V_{\quat}^{R}}+\langle \phi\mid \omega\rangle_{V_{\quat}^{R}}$
	\item[(iv)]
	$\langle \phi\mid \psi\qu\rangle_{V_{\quat}^{R}}=\langle \phi\mid \psi\rangle_{V_{\quat}^{R}}\qu$
	\item[(v)]
	$\langle \phi\qu\mid \psi\rangle_{V_{\quat}^{R}}=\overline{\qu}\langle \phi\mid \psi\rangle_{V_{\quat}^{R}}$
\end{enumerate}
where $\overline{\qu}$ stands for the quaternionic conjugate. It is always assumed that the
space $V_{\quat}^{R}$ is complete under the norm given above and separable. Then,  together with $\langle\cdot\mid\cdot\rangle$ this defines a right quaternionic Hilbert space. Quaternionic Hilbert spaces share many of the standard properties of complex Hilbert spaces. All the spaces considered in this manuscript are right quaternionic Hilbert spaces.

The next two propositions can be established following the proof of their complex counterparts.
\begin{proposition}\label{P-1}\cite{ghimorper,BT}
Let $\mathcal{O}=\{\varphi_{k}\,\mid\,k\in N\}$
be an orthonormal subset of $V_{\quat}^{R}$, where $N$ is a countable index set. Then following conditions are pairwise equivalent:
\begin{itemize}
\item [(a)] The closure of the linear combinations of elements in $\mathcal O$ with coefficients on the right is $V_{\quat}^{R}$.
\item [(b)] For every $\phi,\psi\in V_{\quat}^{R}$, the series $\sum_{k\in N}\langle\phi\mid\varphi_{k}\rangle_{V_{\quat}^{R}}\langle\varphi_{k}\mid\psi\rangle_{V_{\quat}^{R}}$ converges absolutely and it holds:
$$\langle\phi\mid\psi\rangle_{V_{\quat}^{R}}=\sum_{k\in N}\langle\phi\mid\varphi_{k}\rangle_{V_{\quat}^{R}}\langle\varphi_{k}\mid\psi\rangle_{V_{\quat}^{R}}.$$
\item [(c)] For every  $\phi\in V_{\quat}^{R}$, it holds:
$$\|\phi\|^{2}_{V_{\quat}^{R}}=\sum_{k\in N}\mid\langle\varphi_{k}\mid\phi\rangle_{V_{\quat}^{R}}\mid^{2}.$$
\item [(d)] $\mathcal{O}^{\bot}=\{0\}$.
\end{itemize}
\end{proposition}
\begin{definition}
The set $\mathcal{O}$ as in proposition \ref{P-1} is called a {\em Hilbert basis} for $V_{\quat}^{R}$.
\end{definition}
\begin{proposition}\label{P2}
Every separable quaternionic Hilbert space $V_{\quat}^{R}$ has a Hilbert basis. All the Hilbert bases of $V_{\quat}^{R}$ have the same cardinality.

Furthermore, if $\mathcal{O}$ is a Hilbert basis of $V_{\quat}^{R}$, then every  $\phi\in V_{\quat}^{R}$ can be uniquely decomposed as follows:
$$\phi=\sum_{k\in N}\varphi_{k}\langle\varphi_{k}\mid\phi\rangle_{V_{\quat}^{R}},$$
where the series $\sum_{k\in N}\varphi_k\langle\varphi_{k}\mid\phi\rangle_{V_{\quat}^{R}}$ converges absolutely in $V_{\quat}^{R}$.
\end{proposition}

It should be noted that once a Hilbert basis is fixed, every left (resp. right) quaternionic Hilbert space also becomes a right (resp. left) quaternionic Hilbert space \cite{ghimorper,Vis}.

The field of quaternions $\quat$ itself can be turned into a left quaternionic Hilbert space by defining the inner product $\langle \qu \mid \qu^\prime \rangle = \qu \overline{\qu^{\prime}}$ or into a right quaternionic Hilbert space with  $\langle \qu \mid \qu^\prime \rangle = \overline{\qu}\qu^\prime$.
%%%%%%%%%%%%%%%%%%%%%%%%%%%%%%%%%%%%%%%%%%%%%%%%%%%%%%%%%%%%%%%%
\begin{proposition}\cite{Gra1}\label{P1}
	For any non-real quaternion $\qu\in \quat\setminus\mathbb{R}$, there exist, and are unique, $x,y\in\mathbb{R}$ with $y>0$, and $I\in\mathbb{S}$ such that $\qu=x+yI$.
\end{proposition}
\begin{definition}(Slice-regular functions \cite{Jo,S,Gra2})\label{D2}
	Let $\Omega$ be a domain in $\quat$. A real differentiable (i.e., with respect to $x_0$ and the $x_i,\; i=1,2,3$) operator-valued function $f:\Omega\longrightarrow \vr$ is said to be slice right regular if, for every quaternion $I\in\mathbb{S}$, the restriction of $f$ to the complex plane $L_I=\mathbb{R}+I\mathbb{R}$ passing through the origin, and containing $1$ and $I$, has continuous partial derivatives (with
	respect to $x$ and $y$, every element in $L_I$ being uniquely expressible as $x + yI$) and satisfies
	\begin{equation}
		\overline{\partial}_I f (x + yI) := \frac 12\left(\frac {\partial f_I (x + yI )}{\partial x}
		+  \frac {\partial f_I (x + yI )}{\partial y}I\right) = 0\; ,
		\label{rightslicereg}
	\end{equation}
	where $f_I=f|_{\Omega\cap L_I}$.
\end{definition}

With this definition all monomials of the form $\phi\qu^n,~\phi\in \vr,~n\in\mathbb{N}$, are slice right regular. Since regularity respects addition, all polynomials of the form $f(\qu)=\sum_{i=0}^{n}\phi_i\qu^i$, with $\phi_i\in \vr$, are slice right regular. Further, an analog of Abel's theorem guarantees convergence of appropriate infinite power series.
\begin{definition}\cite{Jo}
Let $f:\Omega\subseteq\quat\longrightarrow\vr$ and $\qu=x+yI\in\Omega.$ If $\qu$ is not real then  we say that $f$ admits right-slice derivative in a non-real point $\qu$ if 
$$\partial_Sf(\qu)=\lim_{\pu\rightarrow\qu, \pu\in L_I}(f_I(\pu)-f_I(\qu))(\pu-\qu)^{-1}$$
exists and finite for any $I\in\mathbb{S}$.\\
\end{definition}
Under the above definition the slice derivative of a regular function is regular. For $\phi_n\in\vr$ we have
\begin{equation}\label{E11}
	\partial_S\left(\sum_{n=0}^{\infty}\phi_n \qu^n\right)=\sum_{n=0}^{\infty} n \phi_n \qu^{n-1}.
\end{equation}
The following theorem gives the quaternionic version of holomorphy via a Taylor series. Let $B_\quat(0,r)$ be an open ball in $\quat$, of radius $r>0$ and center at $0$.
\begin{theorem}\cite{Jo,Gra2}\label{T5}
	A function $f: B_\quat(0,r)\longrightarrow \vr$ is right regular if and only if it has a series expansion of the form
	$$ f(\qu)=\sum_{n=0}^{\infty}\frac{1}{n!}\frac{\partial^n f}{\partial x^n}(0)\qu^n$$
	converging on $B_\quat(0, r)$.
\end{theorem}
\begin{remark}
	In general slice-regular functions are not continuous \cite{GSS}. However, under certain assumptions slice continuity can be obtained, see definition 2.7 in \cite{Col}, and even it can be assumed if necessary \cite{Col,Jo}. In this regard, in this manuscript,  we assume continuity for a right regular function wherever needed and still call them simply right regular function.
\end{remark}
%%%%%%%%%%%%%%%%%%%%%%%%%%%%%%%%%%%%%%%%%%%%%%%%%%%%%%%%%%%%%%%%%%%%%%%%%%%%%%%%%%
\section{Right quaternionic linear  operators and some basic properties}
In this section we shall define right  $\quat$-linear operators and recall some basis properties. Most of them are very well known. In this manuscript, we follow the notations in \cite{AC} and \cite{ghimorper}. We shall also recall some results pertinent to the development of the paper. 
\begin{definition}
A mapping $A:\D(A)\subseteq V_{\quat}^R \longrightarrow U_{\quat}^R$, where $\D(A)$ stands for the domain of $A$, is said to be right $\quat$-linear operator or, for simplicity, right linear operator, if
$$A(\phi\qu+\psi\pu)=(A\phi)\qu+(A\psi)\pu,~~\mbox{~if~}~~\phi,\,\psi\in \D(A)~~\mbox{~and~}~~\qu,\pu\in\quat.$$
\end{definition}
The set of all right linear operators from $V_{\quat}^{R}$ to $U_{\quat}^{R}$ will be denoted by $\mathcal{L}(V_{\quat}^{R},U_{\quat}^{R})$ and the identity linear operator on $V_{\quat}^{R}$ will be denoted by $\Iop$. For a given $A\in \mathcal{L}(V_{\quat}^{R},U_{\quat}^{R})$, the range and the kernel will be
\begin{eqnarray*}
\text{ran}(A)&=&\{\psi \in U_{\quat}^{R}~|~A\phi =\psi \quad\text{for}~~\phi \in\D(A)\}\\
\ker(A)&=&\{\phi \in\D(A)~|~A\phi =0\}.
\end{eqnarray*}
We call an operator $A\in \mathcal{L}(V_{\quat}^{R},U_{\quat}^{R})$ bounded (or continuous) if
\begin{equation}\label{PE1}
\|A\|=\sup_{\|\phi \|_{\vr}=1}\|A\phi \|_{\ur}<\infty,
\end{equation}
or equivalently, there exist $K\geq 0$ such that $\|A\phi \|_{\ur}\leq K\|\phi \|_{\vr}$ for all $\phi \in\D(A)$. The set of all bounded right linear operators from $V_{\quat}^{R}$ to $U_{\quat}^{R}$ will be denoted by $\B(V_{\quat}^{R},U_{\quat}^{R})$. The set of all bounded right linear operators from $V_{\quat}^{R}$ to $V_{\quat}^{R}$ will be denoted by $\B(V_{\quat}^{R})$. Set of all  invertible bounded right linear operators from $V_{\quat}^{R}$ to $U_{\quat}^{R}$ will be denoted by $\mathcal{G} (V_{\quat}^{R},U_{\quat}^{R})$. We also denote for a set $\Delta\subseteq\quat$, $\Delta^*=\{\oqu~|~\qu\in\Delta\}$.
\\
Assume that $V_{\quat}^{R}$ is a right quaternionic Hilbert space, $A$ is a right linear operator acting on it.
Then, there exists a unique linear operator $A^{\dagger}$ such that
\begin{equation}\label{Ad1}
\langle \psi \mid A\phi \rangle_{\ur}=\langle A^{\dagger} \psi \mid\phi \rangle_{\vr};\quad\text{for all}~~~\phi \in \D (A), \psi\in\D(A^\dagger),
\end{equation}
where the domain $\D(A^\dagger)$ of $A^\dagger$ is defined by
$$
\D(A^\dagger)=\{\psi\in U_{\quat}^{R}\ |\ \exists \varphi\ {\rm such\ that\ } \langle \psi \mid A\phi \rangle_{\ur}=\langle \varphi \mid\phi \rangle_{\vr}\}.$$
The following theorem gives two important and fundamental results about right $\quat$-linear bounded operators which are already appeared in \cite{ghimorper} for the case of $\vr=\ur$. Point (b) of the following theorem is known as the open mapping theorem.
\begin{theorem}\cite{Fr}\label{open} Let $A:\D(A)\subseteq V_{\quat}^R \longrightarrow U_{\quat}^R$ be a right $\quat$-linear operator. Then%\cite{ghimorper}
\begin{itemize} 
\item[(a)] $A\in\B(\vr,\ur)$ if and only if $A$ is continuous.
\item[(b)] if $A\in\B(\vr,\ur)$ is surjective, then $A$ is open. In particular, if $A$ is bijective then $A^{-1}\in\B(\vr,\ur)$.
\end{itemize}
\end{theorem}
The following proposition provides some useful aspects about the orthogonal complement subsets.
\begin{proposition}\cite{Fr}\label{ort}
Let $M\subseteq\vr$. Then
\begin{itemize}
\item [(a)] $M^{^\perp}$ is closed.
\item [(b)] if $M$ is a closed subspace of $\vr$ then $\vr=M\oplus M^\perp$.
\item [(c)] if $\dim(M)<\infty$, then $M$ is a closed subspace.
\end{itemize}
\end{proposition}
\begin{proposition} \cite{ghimorper, Fr} \label{IP30}
Let $A\in\B(\vr, \ur)$. Then
\begin{itemize}
\item [(a)] $\ra(A)^\perp=\kr(A^\dagger).$
\item [(b)] $\kr(A)=\ra(A^\dagger)^\perp.$
\item [(c)] $\kr(A)$ is closed subspace of $\vr$.
\end{itemize}
\end{proposition}
\begin{proposition}\label{dag}\cite{ghimorper} $A\in\B(\vr)$, then $A^\dagger\in\B(\vr), \|A\|=\|A^\dagger\|$ and $\|A^\dagger A\|=\|A\|^2$.
\end{proposition}
\begin{definition}\cite{Ai}\label{BBD}
An operator $A\in\B(\vr)$ is said to be bounded below if $A$ is injective and has closed range.
\end{definition}
\begin{proposition}\label{BBP}
$A\in\B(\vr)$ is bounded below if and only if there exists $K>0$ such that $\|A\phi\|\geq K\|\phi\|$ for all $\phi\in\vr$.
\end{proposition}
\begin{proof}
	A proof follows exactly as a complex proof. For a complex proof see \cite{Ai}, page 15.	
\end{proof}
%%%%%%%%%%%%%%%%
\begin{proposition}\label{BBP2}
	Let $A\in\B(\vr)$. Then $A^2$ is bounded below if and only if $A$ is bounded below (hence $A^n$ is bounded below for any $n\in\N$ if and only if $A$ is bounded below).
\end{proposition}
\begin{proof}
	Suppose $A$ is bounded below. Then $\kr(A)=\{0\}$ and $A(\vr)$ is closed. Since $\kr(A^2)\subseteq\kr(A)$ and the image of a closed set under continuous map is closed, $A^2$ is bounded below.
	Conversely, suppose $A^2$ is bounded below. Then $\kr(A^2)=\{0\}$ and $A^2(\vr)$ is closed. Let $\phi\in\vr$ and $A(\phi)=0$, then $A^2(\phi)=A(0)=0$ thus $\phi=0$, and hence $A$ is injective. Let $\{\phi_n\}\subseteq A(\vr)$ such that $\phi_n\longrightarrow\phi$ as $n\rightarrow\infty$, then  $A(\phi_n)\longrightarrow A(\phi)$ as $n\rightarrow\infty$. Therefore $A(\phi)\in A^2(\vr)$ and hence $\phi\in A(\vr)$. Therefore, $A(\vr)$ is closed.
\end{proof}
%\begin{proposition}\label{IP4}
%Let $A:\D(A):\vr\longrightarrow\ur$ be a right quaternionic linear operator. %If $A$ is closed and satisfies the condition that there exists $c>0$ such that  \begin{center}
%$\|A\phi\|_{\ur}\geq c\|\phi\|_{\vr}$, for all $\phi\in\D(A)$,
%\end{center} then $\text{ran}(A)$ is closed.
%\end{proposition}
%\begin{proof}
%The proof can be manipulated from the proof of proposition 2.13 in \cite{BT}.
%\end{proof}

\begin{theorem}\cite{Fr}(Bounded inverse theorem)\label{NT1}
Let $A\in\B(\vr,\ur)$, then the following results are equivalent.
\begin{enumerate}
\item [(a)] $A$ has a bounded inverse on its range.
\item[(b)] $A$ is bounded below.
\item[(c)] $A$ is injective and has a closed range.
\end{enumerate}
\end{theorem}
\begin{proposition}\cite{Fr}\label{NP1}
Let $A\in\B(\vr,\ur)$, then $\ra(A)$ is closed in $\ur$ if and only if $\ra(A^\dagger)$ is closed in $\vr$.
\end{proposition}

\begin{proposition}\label{ST}
	Let $A,B\in\B(\vr)$. Assume that $AB=BA$. Then $AB$ is invertible if and only if both $A$ and $B$ are invertible.
\end{proposition}
\begin{proof}
	A proof follows its complex counterpart. For a complex proof see \cite{Fabi}, page 213.
\end{proof}
\begin{definition}\cite{La}\label{ID1}
Let $A\in\B(\vr)$. A closed subspace $M\subseteq\vr$ is said to be $A$-invariant  if $A(M)\subseteq M$, where $A(M)=\{A\phi~|~\phi\in M\}$. It is said to be $A$-hyperinvariant if $B(M)\subseteq M$ for every $B\in\B(\vr)$ that commutes with $A$.
\end{definition}
If $A\in\B(\vr)$, in order to be compatible with the inner product in $\vr$, the scalar multiplication of $A$ is defined as
$$(\qu A)(\phi)=A(\phi)\qu, \quad \qu\in\quat.$$

%%%%%%%%%%%%%%%%%%%%%%%%%%%%%%%%%%%%%%%%%%%%%%%%%%
%%%%%%%%%%%%%%%%%%%%%%%%%%%%%%%%%%%%%%%%%%%%%%%%%%
\subsection{S-Spectrum}
For a given right linear operator $A:\D(A)\subseteq V_{\quat}^R\longrightarrow V_{\quat}^R$ and $\qu\in\quat$, we define the operator $R_{\qu}(A):\D(A^{2})\longrightarrow\quat$ by  $$R_{\qu}(A)=A^{2}-2\text{Re}(\qu)A+|\qu|^{2}\Iop,$$
where $\qu=q_{0}+\bi q_1 + \bj q_2 + \bk q_3$ is a quaternion, $\text{Re}(\qu)=q_{0}$  and $|\qu|^{2}=q_{0}^{2}+q_{1}^{2}+q_{2}^{2}+q_{3}^{2}.$\\
In the literature, the operator is called pseudo-resolvent since it is not the resolvent operator of $A$ but it is the one related to the notion of spectrum as we shall see in the next definition. For more information, on the notion of $S$-spectrum the reader may consult e.g. \cite{Fab, Fab1, NFC,Jo}, and  \cite{ghimorper}.
In this setting, for $\qu\in\quat$, we can easily see that
$$R_\qu(A)=A^2-2\text{Re}(\qu)A+|\qu|^2=(A-\qu\Iop)(A-\oqu\Iop)=(A-\oqu\Iop)(A-\qu\Iop),$$
where $R_\qu(A)$ is linear in $\vr$ while $A-\qu\Iop$ and $A-\oqu\Iop$ are not linear in $\vr$.
\begin{definition}
Let $A:\D(A)\subseteq V_{\quat}^R\longrightarrow V_{\quat}^R$ be a right linear operator. The {\em $S$-resolvent set} (also called \textit{spherical resolvent} set) of $A$ is the set $\rho_{S}(A)\,(\subset\quat)$ such that the three following conditions hold true:
\begin{itemize}
\item[(a)] $\ker(R_{\qu}(A))=\{0\}$.
\item[(b)] $\text{ran}(R_{\qu}(A))$ is dense in $V_{\quat}^{R}$.
\item[(c)] $R_{\qu}(A)^{-1}:\text{ran}(R_{\qu}(A))\longrightarrow\D(A^{2})$ is bounded.
\end{itemize}
The \textit{$S$-spectrum} (also called \textit{spherical spectrum}) $\sigma_{S}(A)$ of $A$ is defined by setting $\sigma_{S}(A):=\quat\smallsetminus\rho_{S}(A)$. For a bounded linear operator $A$ we can write the resolvent set as
\begin{eqnarray*}
\rho_S(A)&=& \{\qu\in\quat~|~R_\qu(A)\in\mathcal{G}(V_{\quat}^R)\}\\
&=&\{\qu\in\quat~|~R_\qu(A)~\text{has an inverse in}~\B(V_{\quat}^R)\}\\
&=&\{\qu\in\quat~|~\text{ker}(R_\qu(A))=\{0\}\quad\text{and}\quad \text{ran}(R_\qu(A))=V_\quat^R\}
\end{eqnarray*}
and the spectrum can be written as
\begin{eqnarray*}
\sigma_S(A)&=&\quat\setminus\rho_S(A)\\
&=&\{\qu\in\quat~|~R_\qu(A)~\text{has no inverse in}~\B(V_{\quat}^R)\}\\
&=&\{\qu\in\quat~|~\text{ker}(R_\qu(A))\not=\{0\}\quad\text{or}\quad \text{ran}(R_\qu(A))\not=V_\quat^R\}.
\end{eqnarray*}
The spectrum $\sigma_S(A)$ decomposes into three major disjoint subsets as follows:
\begin{itemize}
\item[(i)] the \textit{spherical point spectrum} of $A$: $$\sigma_{pS}(A):=\{\qu\in\quat~\mid~\ker(R_{\qu}(A))\ne\{0\}\}.$$
\item[(ii)] the \textit{spherical residual spectrum} of $A$: $$\sigma_{rS}(A):=\{\qu\in\quat~\mid~\ker(R_{\qu}(A))=\{0\},\overline{\text{ran}(R_{\qu}(A))}\ne V_{\quat}^{R}~\}.$$
\item[(iii)] the \textit{spherical continuous spectrum} of $A$: $$\sigma_{cS}(A):=\{\qu\in\quat~\mid~\ker(R_{\qu}(A))=\{0\},\overline{\text{ran}(R_{\qu}(A))}= V_{\quat}^{R}, R_{\qu}(A)^{-1}\notin\B(V_{\quat}^{R}) ~\}.$$
\end{itemize}
If $A\phi=\phi\qu$ for some $\qu\in\quat$ and $\phi\in V_{\quat}^{R}\smallsetminus\{0\}$, then $\phi$ is called an \textit{eigenvector of $A$ with right eigenvalue} $\qu$. The set of right eigenvalues coincides with the point $S$-spectrum, see \cite{ghimorper}, proposition 4.5.
\end{definition}
Note also that the function $\qu\rightarrow R_\qu(A)$ is continuous and $R_\qu(A)^{-1}$ is continuous on $\rho_S(A)$ \cite{Jo}.
\begin{proposition}\cite{Fab2, ghimorper}\label{PP1}
For $A\in\B(\vr)$, the resolvent set $\rho_S(A)$ is a non-empty open set and the spectrum $\sigma_S(A)$ is a non-empty compact set.
\end{proposition}
\begin{remark}\label{R1}
For $A\in\B(\vr)$, since $\sigma_S(A)$ is a non-empty compact set so is its boundary. That is, $\partial\sigma_S(A)=\partial\rho_S(A)\not=\emptyset$.
\end{remark}

%%%%%%%%%%%%%%%%%%%%%%%%%%%%%%%%%%%%%%%%%%%%%%%%%%%%%%%%%%%%%%%%%
%%%%%%%%%%%%%%%%%%%%%%%%%%%%%%%%%%%%%%%%%%%%%%%
\section{Surjectivity $S$-spectrum and Approximate $S$-point spectrum}
Following the complex case, for $A\in\B(\vr)$, the approximate S-point spectrum was studied in \cite{Fr}. We recall the definition and some results from \cite{Fr} as needed here. Then we define and study the surjectivity S-spectrum, in the quaternionic setting, following its complex counterpart. For the theory of complex surjectivity spectrum we refer the reader to \cite{Ai, La}.
\begin{definition}\cite{Fr}\label{D1}
Let $A\in\B(V_\quat^R)$. The {\em approximate S-point spectrum} of $A$, denoted by $\apo(A)$, is defined as
$$\apo(A)=\{\qu\in\quat~~|~~\text{there is a sequence}~~\{\phi_n\}_{n=1}^{\infty}~~\text{such that}~~\|\phi_n\|=1~~\text{and}~~\|R_\qu(A)\phi_n\|\longrightarrow 0\}.$$
\end{definition}
\begin{proposition}\label{P3}\cite{Fr}
Let $A\in\B(\vr)$, then $\sigma_{pS}(A)\subseteq\apo(A)$. 
\end{proposition}
\begin{proposition}\label{P4}\cite{Fr}
If $A\in\B(\vr)$ and $\qu\in\quat$, then the following statements are equivalent.
\begin{enumerate}
\item[(a)] $\qu\not\in\apo(A).$
\item[(b)] $\text{ker}(R_\qu(A))=\{0\}$ and $\text{ran}(R_\qu(A))$ is closed.
\item[(c)] There exists a constant $c\in\R$, $c>0$ such that $\|R_\qu(A)\phi\|\geq c\|\phi\|$ for all $\phi\in\D(A^2)$.
\end{enumerate}
\end{proposition}
\begin{theorem}\label{T1}\cite{Fr}Let $A\in\B(\vr)$, then $\apo(A)$ is a non-empty closed subset of $\quat$ and $\partial\sigma_S(A)\subseteq\apo(A),$ where $\partial\sigma_S(A)$ is the boundary of $\sigma_S(A)$.
\end{theorem}
%%%%%%%%%%%%%
\begin{theorem}\label{T2}\cite{Fr}
Let $A\in\B(\vr)$ and $\qu\in\quat$, then the following statements are equivalent.
\begin{enumerate}
\item [(a)]$\qu\not\in\apo(A)$.
\item[(d)]$\text{ran}(R_{\oqu}(A^\dagger))=\vr.$
\end{enumerate}
\end{theorem}
%%%%%%
\begin{proposition}\label{C1}\cite{Fr}
If $A\in\B(\vr)$, then $\partial\sigma_S(A)\subseteq\apo(A)\cap\apo(A^\dagger)^*$.
\end{proposition}
%%%%%%%%%%%
Following the complex formalism in the following we define the S-compression spectrum for an operator $A\in\B(\vr)$.
\begin{definition}\label{DC}
The spherical compression spectrum of an operator $A\in\B(\vr)$, denoted by $\sigma_c^S(A)$, is defined as
$$\sigma_c^S(A)=\{\qu\in\quat~~|~~\text{ran}(R_\qu(A))~~~\text{is not dense in}~~\vr~\}.$$
\end{definition}
\begin{proposition}\label{CP1}\cite{Fr}
Let $A\in\B(\vr)$ and $\qu\in\quat$. Then,
\begin{enumerate}
\item[(a)] $\qu\in\sigma_c^S(A)$ if and only if $\oqu\in\sigma_{pS}(A)$.
\item[(b)] $\sigma_S(A)=\apo(A)\cup\sigma_c^S(A)$.
\end{enumerate}
\end{proposition}
Since the S-surjectivity spectrum and its connection to other parts of the spectrum have not been addressed yet, we shall define it and study some of its properties according to \cite{La}. Later we shall also investigate its connection to Kato S-spectrum and local S-spectrum.
\begin{definition}\label{su} Let $A\in\B(\vr)$. The surjectivity S-spectrum of $A$ is defined as
	$$\sus(A)=\{\qu\in\quat~|~\text{ran}(R_{\qu}(A)\not=\vr\}.$$
\end{definition}
Clearly we have
\begin{equation}\label{sue1}
\sigma_c^S(A)\subseteq\sus(A)\quad\text{and}\quad\sigma_S(A)=\sigma_{pS}(A)\cup\sus(A).
\end{equation}
%%%%%%%%%
\begin{proposition}\label{su1} Let $A\in\B(\vr)$. Then $A$ has the following properties.
	\begin{enumerate}
		\item[(a)]$\sigma_{pS}(A)\subseteq\sigma_c^S(A^\dagger)~~\text{and}~~\sigma_c^S(A)=\sigma_{pS}(A^\dagger)$.
		\item[(b)]$\sus(A)=\apo(A^\dagger)~~\text{and}~~\apo(A)=\sus(A^\dagger).$
		\item[(c)]$\sigma_S(A)=\sigma_S(A^\dagger).$
	\end{enumerate}
\end{proposition}
\begin{proof} (a)~~
Let $\qu\in\quat\setminus\sigma_c^S(A^\dagger)$, then $\text{ran}(R_\qu(A^\dagger))$ is dense in $\vr$.	From proposition \ref{IP30} we have $\kr(R_\qu(A))=\text{ran}(R_\qu(A^\dagger))^\perp$. Let $\phi\in\kr(R_\qu(A))$ and let $\psi\in\overline{\text{ran}(R_\qu(A^\dagger))}=\vr$. Then there exists a sequence $\{\psi_n\}\subseteq\text{ran}(R_\qu(A^\dagger))$ such that $\psi_n\longrightarrow\psi$ as $n\longrightarrow\infty$. Further, since $\langle\phi|\psi_n\rangle=0$ for all $n$, we have $\langle\phi|\psi\rangle=0$. That is, $\langle\phi|\psi\rangle=0$ for all $\psi\in\vr$, and hence $\phi=0$. Therefore $\kr(R_\qu(A))=\{0\}$ and which implies $\qu\in\quat\setminus\sigma_{pS}(A)$. Thus $\sigma_{pS}(A)\subseteq\sigma_c^S(A^\dagger)$.\\
By the preceding paragraph, $\sigma_{ps}(A^\dagger)\subseteq\sigma_{c}^S(A)$. To see other inclusion, take $\qu\notin\sigma_{ps}(A^\dagger)$. Then $\ker(R_\qu(A^\dagger))=\text{ran}(R_\qu(A))^\perp=\{0\}$. This implies $\overline{\text{ran}(R_\qu(A))}=\vr$. Thus $\qu\notin\sigma_{c}^S(A)$. 
%Since $R_{\oqu}(A)=R_\qu(A)$ and, by proposition \ref{dag}, $\|R_\qu(A)\|=\|R_\qu(A)^\dagger\|=\|R_{\oqu}(A^\dagger)\|$, the equality $\sigma_c^S(A)=\sigma_{pS}(A^\dagger)$ follows from part(a) of proposition \ref{CP1}.\\
\\(b)~Given any $\qu\in\quat\setminus\sus(A)$, we have $\ra(R_\qu(A))=\vr$. Since, by proposition \ref{IP30}, $\ra(R_\qu(A))^\perp=\kr(R_\qu(A)^\dagger)$, we get $\kr(R_\qu(A^\dagger))=\{0\}$. Therefore, by the bounded inverse theorem, $R_\qu(A^\dagger)^{-1}$ is bounded, so $R_\qu(A^\dagger)$ is bounded below. Therefore, by proposition \ref{P4}, we have $\qu\in\quat\setminus\apo(A^\dagger)$, and hence $\apo(A^\dagger)\subseteq\sus(A)$. Conversely, let $\qu\notin\apo(A^\dagger)$ then by proposition \ref{P4} we have $\kr(R_\qu(A^\dagger))=\{0\}$ and $\ra(R_\qu(A^\dagger))$ is closed. Therefore, $\ra(R_\qu(A))=\kr(R_\qu(A^\dagger))^\perp=\{0\}^\perp=\vr$. Thus $\qu\notin\sus(T)$, and hence $\sus(A)\subseteq\apo(A^\dagger)$. All together we get $\sus(A)=\apo(A^\dagger)$.\\
For the second equality, let $\qu\notin\apo(A)$, then by proposition \ref{P4} we have $\kr(R_\qu(A))=\{0\}$ and $\ra(R_\qu(A))$ is closed. Therefore, by proposition \ref{NP1}, $\ra(R_\qu(A^\dagger))$ is closed, and also by proposition \ref{IP30}, $\vr=\{0\}^\perp=\kr(R_\qu(A))^\perp=\ra(R_\qu(A^\dagger))$. Thus $\qu\notin\sus(A^\dagger)$, hence $\sus(A^\dagger)\subseteq\apo(A)$. For the other inclusion, let $\qu\notin\sus(A^\dagger)$, then $\ra(R_\qu(A^\dagger))=\vr$. By proposition \ref{IP30}, $\ra(R_\qu(A^\dagger))^\perp=\kr(R_\qu(A))=\{0\}$. Since $\ra(R_\qu(A^\dagger))$ is closed, by proposition \ref{NP1}, $\ra(R_\qu(A))$ is closed. Therefore, by proposition \ref{P4}, $\qu\notin\apo(A)$, and hence $\apo(A)\subseteq\sus(A^\dagger)$. Thus $\apo(A)=\sus(A^\dagger)$.\\
(c)~~From part(b) of proposition \ref{CP1},  above parts (a),(b) and equation \ref{sue1}, we get
$$\sigma_S(A)=\apo(A)\cup\sigma_c(A)=\sus(A^\dagger)\cup\sigma_{pS}(A^\dagger)=\sigma_S(A^\dagger).$$
\end{proof}
%%%%%%%%%%%%%%%%%%%%%%%%%%%%%
\begin{proposition}\label{su2}
For $A\in\B(\vr)$, $\sus(A)$ is closed and $\partial\sigma_S(A)\subseteq\sus(A).$
\end{proposition}
\begin{proof}
Let $A\in\B(\vr),$ then by proposition \ref{dag}, $A^\dagger\in\B(\vr)$. Therefore, by theorem \ref{T1}, $\apo(A^\dagger)$ is closed and $\partial\sigma_S(A^\dagger)\subseteq\apo(A^\dagger).$ By proposition \ref{su1}, $\sus(A)=\apo(A^\dagger)$ and $\sigma_S(A)=\sigma_S(A^\dagger)$. Hence $\sus(A)$ is closed and $\partial\sigma_S(A)\subseteq\sus(A).$
\end{proof}
%%%%%%%%%%%%%%%%%%%%
\begin{proposition}\label{su3}
Let $A\in\B(\vr)$ and $M$, $N$ be two closed $A$-invariant subspaces of $\vr$ such that $\vr=M\oplus N$. Then
\begin{enumerate}
	\item [(a)] $\apo(A)=\apo(A|_M)\cup\apo(A|_N)$;
	\item[(b)] $\sus(A)=\sus(A|_M)\cup\sus(A|_N)$;\
	\item[(c)] $\sigma_S(A)=\sigma_S(A|_M)\cup\sigma_S(A|_N)$.
\end{enumerate}	
\end{proposition}
\begin{proof}
	(a)~~Let $P_M:\vr\longrightarrow M$ be the projection operator. Clearly $P_M$ commutes with $A$. It is easily seen that $\kr(A)=\kr(A|_M)\oplus\kr(A|_N)$ and $A(\vr)=A(M)\oplus A(N)$. Thus, $A$ is injective if and only if $A|_M$ and $A_N$ are injective.\\
	{\em Claim:} $A(\vr)$ is closed if and only if $A(M)$ and $A(N)$ are closed in $M$ and $N$ respectively.\\
	If $A(\vr)$ is closed, then $A(M)=AP_M(\vr)=P_M(A(\vr))=A(\vr)\cap M$. Therefore $A(M)$ is closed in $M$. Similarly $A(N)$ is closed in $N$. Conversely, assume that $A(M)$ is closed in $M$ and $A(N)$ is closed in $N$. Since the mapping $\Psi:M\times N\longrightarrow M\oplus N$ defined by $\Psi((\phi,\psi))=\phi+\psi$ is a topological isomorphism, then the image $\Psi(A(M)\times A(N))=A(M)\oplus A(N)=A(\vr)$ is closed in $\vr$. Thus, combining the above results: $A$ is bounded below if and only if $A|_M$ and $A|_N$ are bounded below. As a consequence, $R_\qu(A)$ is bounded below if and only if $R_\qu(A)|_M$ and $R_\qu(A)|_N$ are bounded below. Hence (a) is proved.\\
	(b)~Similarly using $A(\vr)=A(M)\oplus A(N)$ we can easily show that $A$ is onto if and only if $A|_M$ and $A|_N$ are onto. Consequently, $R_\qu(A)$ is onto if and only if $R_\qu(A)|_M$ and $R_\qu(A)|_N$ are onto, which proves (b).\\
	(c)~~From the above arguments it is clear that  $R_\qu(A)$ is bijective if and only if $R_\qu(A)|_M$ and $R_\qu(A)|_N$ are bijective, which proves (c).
	
\end{proof}
%%%%%%%%%%%%%%%%%%%%%%
\begin{proposition}\label{1.6.1}
	(\cite{La}, page 76) Let $\{a_n\}_{n\in\N}$ be a sequence of positive real numbers that is sub-multiplicative, in the sense that $a_{m+n}\leq a_ma_n$ for all $m,n\in\N$. Then
	$$a_n^{1/n}\longrightarrow \inf\{a_k^{1/k}~~|~~k\in\N\}\quad\text{as}~~n\rightarrow\infty.$$
	Similarly if the sequence $\{a_n\}_{n\in\N}$ satisfies $a_{m+n}\geq a_ma_n$ for all $m,n\in\N$,
	then $$a_n^{1/n}\longrightarrow \sup\{a_k^{1/k}~~|~~k\in\N\}\quad\text{as}~~n\rightarrow\infty.$$
\end{proposition}
\begin{remark}\label{BR1}
	Let $\vr$ is non-trivial and $A\in\B(\vr)$.
	\begin{enumerate}
		\item [(a)] From proposition \ref{1.6.1}, the S-spectral radius $r_S(A)=\max\{|\qu|~~:~~\qu\in\sigma_S(A)\}$   is
		$$r_S(A)=\lim_{n\rightarrow\infty}\|A^n\|^{1/n}=\inf_{n\in\N}\|A^n\|^{1/n}.$$
		\item[(b)] From Proposition \ref{1.6.1} we can also handle the lower bounds: if $$\kappa(A)=\inf\{\|A\phi\|~~|~~\phi\in\vr~~\text{with}~~\|\phi\|=1\}$$  denotes the lower bound of $A$, then $\kappa(A^m)\kappa(A^n)\leq\kappa(A^{m+n})$ for all $m,n\in\N$.
		\item[(c)] $\kappa(A)=0$ whenever $\kappa(A^n)=0$ for some $n\in\N$.
		\item[(d)] By proposition \ref{BBP2}, if $\kappa(A)=0,$ then $0\in\apo(A)$, and hence $\kappa(A^n)=0$ for all $n\in\N$.
		\item[(e)] If $A$ is invertible then $\kappa(A)=\|A^{-1}\|^{-1}$.
		\item[(f)] Proposition \ref{1.6.1} ensures the existence of the limit
		$$i(A)=\lim_{n\rightarrow\infty}\kappa(A^n)^{1/n}=\sup_{n\in\N}\kappa(A^n)^{1/n}.$$
		It is immediate that $i(A)\leq r_S(A).$
		\item[(g)]		Let $M>0$ and $c>0$, and $\qu=q_0+q_1i+q_2j+q_3k\in\quat$ and also denote $\beta_n(M,\qu)=(2|\text{Re}(\qu^n)|M^n+|\qu|^{2n})^{\frac{1}{2n}}$, then
		\begin{eqnarray*}
			c^{2n}-2|\text{Re}(\qu^n)|M^n-|\qu|^{2n}>0&\Leftrightarrow&c^{2n}>2|\text{Re}(\qu^n)|M^n+|\qu|^{2n}\\
			&\Leftrightarrow&c>(2|\text{Re}(\qu^n)|M^n+|\qu|^{2n})^{\frac{1}{2n}}=\beta_n(M,\qu).
		\end{eqnarray*}
		Also note that $\beta_n(M,\qu)\geq |\qu|$ and $\beta_n(M,\qu)>0$ if $\qu\not=0$.
	\end{enumerate}
\end{remark}
In the following $\nabla_\quat(\qu,r):=\{\pu\in\quat~~|~~|\qu-\pu|\leq r\}$ denotes the closed ball centered at $\qu$ and radius $r\geq 0$. $B_\quat(\qu,r)$ is the open ball with center $\qu$ and radius $r>0$. 
%%%%%%%%%%
\begin{proposition}\label{1.6.2} Every operator $A\in\B(\vr)$ has the following properties.
	\begin{enumerate}
		\item [(a)] $\apo(A)$ is contained in the spherical annulus $\{\qu\in\quat~~|~~i(A)\leq|\qu|\leq r_S(A)\}$.
		\item[(b)] If $A$ is non-invertible, then $\nabla_\quat(0, i(A))\subseteq\sigma_S(A)$.
		\item[(c)]If $A$ is invertible, then $B_\quat(0,i(A))\subseteq\rho_S(A)$.
		\item[(d)] If $A$ in non-invertible and $i(A)=r_S(A)$, then $\sigma_S(A)=\nabla_\quat(0,r_S(A))$.
		\item[(e)] If $A$ is invertible and $i(A)=r_S(A)$, then $\sigma_S(A)=\{\qu\in\quat~|~|\qu|=r_S(A)\}=\partial\nabla_\quat(0,r_S(A))$.
	\end{enumerate}
\end{proposition}
\begin{proof}(a)~~
	Let $A\in\B(\vr)$, then there exist an $M>0$ such that $\|A^n(\phi)\|<M^n\|\phi\|$, for all $n\in\mathbb{N}$.  Clearly $\apo(A)\subseteq\sigma_S(A)\subseteq\nabla(0,r_S(A))$. Thus, it remains to be seen that $\qu\in\quat$ with $|\qu|<i(A)$ cannot belongs to $\apo(A)$. Choose a real number $c>0$ and an integer $n\in\N$ such that $c^n\leq\kappa(A^n)$ and $\beta_n(M,\qu)<c<i(A)$, where $\beta_n(M,\qu)$ is as in part (g) of remark \ref{BR1}. Note that, such a $c$ can be chosen by the supremum property. Since $c^n\leq\kappa(A^n)$, from part (b) of remark \ref{BR1} $c^{2n}\leq\kappa(A^{2n})$. Because $c^{2n}\leq\kappa(A^{2n})$ we have $$c^{2n}\|\phi\|\leq \|A^{2n}\phi\|~~~\text{ for all}~~~ \phi\in\vr.$$
	From the above we obtain
	\begin{eqnarray*}
		\|R_{\qu^n}(A^n)\|&=&\|A^{2n}\phi-2\text{Re}(\qu^n)\phi+|\qu|^{2n}\|\\
		&\geq&\|A^{2n}\phi\|-2|\text{Re}(\qu^n)|\|A^n\phi\|-|\qu|^{2n}\\
		&\geq&(c^{2n}-2|\text{Re}(\qu^n)|M^n-|\qu|^{2n})\|\phi\|\quad\text{for all}~~\phi\in\vr.
	\end{eqnarray*}
	Therefore, by part (g) of remark \ref{BR1} and proposition \ref{P4}, $\qu^n\notin\apo(A^n)$. Now we have
	\begin{eqnarray*}
		R_{\qu^n}(A^n)\phi&=&(A^n-\qu^n)(A^n-\oqu^n)\phi\\
		&=&\left(\sum_{k=1}^{n}\qu^{n-k}A^{k-1}(A-\qu)\right)\left(\sum_{j=1}^{n}\oqu^{n-j}A^{j-1}(A-\oqu)\right)\phi\\
		&=&\sum_{k=1}^{n}\sum_{j=1}^{n}(\qu^{n-k}A^{k-1}(A-\qu))(A^j\phi\oqu^{n-j}-A^{j-1}\phi\oqu^{n-j+1})\\
		&=&\sum_{k=1}^{n}\sum_{j=1}^{n}\qu^{n-k}A^{k-1}(A^{j+1}\phi\oqu^{n-j}-A^j\phi\oqu^{n-j+1}-A^j\phi\qu\oqu^{n-j}+A^{j-1}\phi\qu\oqu^{n-j+1})\\
		&=&\sum_{k=1}^{n}\sum_{j=1}^{n}(A^{j+k}\phi\qu^{n-k}\oqu^{n-j}-A^{j+k-1}\phi\qu^{n-k}\oqu^{n-j+1}-A^{j+k-1}\phi\qu^{n-k+1}\oqu^{n-j}+\\
		&  &\quad\quad\quad\quad\quad\quad\quad\quad+A^{j+k-2}\phi\qu^{n-k+1}\oqu^{n-j+1})\\
		&=&R_\qu(A)\sum_{k=1}^{n}\sum_{j=1}^{n}(\qu^{n-k}\oqu^{n-j}A^{j+k-2})\phi.
	\end{eqnarray*}
	Therefore, by proposition \ref{P4} and part (g) of remark \ref{BR1}, $\qu\notin\apo(A)$ for $|\qu|<i(A)$.\\
	
	(b)~~Let $\qu\in\quat$ for which $|\qu|\leq i(A)$. If $\qu\in\rho_S(A)$, then, since $A$ is not invertible, by proposition \ref{ST}, $0\in\sigma_S(A)$, and $\rho_S(A)$ is open, $t\qu\in\partial\sigma_S(A)$ for some $t\in[0,1)$. Then, by proposition \ref{T1}, we have $t\qu\in\apo(A)$, which contradicts part (a) because $|t\qu|<i(A)$. Hence, $\nabla_\quat(0,i(A))\subseteq\sigma_S(A)$.\\
	
	(c)~~Let $\qu\in\quat$ with $|\qu|<i(A)$, and assume that $\qu\in\sigma_S(A)$. Since $A$ is invertible, by proposition \ref{ST}, $A^2$ is invertible, and hence $0\in\rho_S(A)$. Therefore we can have $|\pu|\leq|\qu|<i(A)$ for some $\pu\in\partial\sigma_S(A)$. Hence, by proposition \ref{T1}, $\pu\in\apo(A)$ which is impossible by part (a). Therefore $\qu\in\rho_S(A)$ for all $\qu\in\quat$ for which $|\qu|<i(A)$, and hence $B_\quat(0,i(A))\subseteq\rho_S(A)$.\\
	
	(d)~~Clearly $\sigma_S(A)\subseteq\nabla_\quat(0,r_s(A))$. Since $A$ is non-invertible and $i(A)=r(A)$, from part (b), we have $\nabla_\quat(0,r_S(A))\subseteq\sigma_S(A)$. Thus $\sigma_S(A)=\nabla_\quat(0,r_S(A))$.\\
	
	(e)~~Clearly $\sigma_S(A)\subseteq\nabla_\quat(0,r_s(A))$. Since $A$ is invertible and $i(A)=r(A)$, from part (c), we have $B_\quat(0,r_S(A))\subseteq\rho_S(A)$. Thus
	$$\sigma_S(A)=\nabla_\quat(0,r_S(A))\cap(\quat\setminus B_\quat(0,r_S(A)))=\{\qu\in\quat~~|~~|\qu|=r_S(A)\}.$$
\end{proof}
\begin{remark}\label{SR1}
	If $A$ is an isometry, that is $\|A(\phi)\|=\|\phi\|$ for all $\phi\in\vr$, then $r_S(A)=i(A)=1$, hence $\apo(A)\subseteq\partial B_\quat(0,1)$, the quaternionic unit sphere. If $A$ is an invertible isometry, then by theorem \ref{T1} and parts (a), (e) of the above proposition $\apo(A)=\sigma_S(A)=\partial B_\quat(0,1)$ while if $A$ is a non-invertible isometry then, by part (d) of the above proposition $\sigma_S(A)=\nabla_\quat(0,1)$.
	\end{remark}

%%%%%%%%%%%%%%%%%%%%%%%
\section{Hyper-kernel and hyper-range of a right linear operator on $\vr$}
Let $A\in\B(\vr)$, then clearly we have
$$\kr(A^0)=\{0\}\subseteq\kr(A)\subseteq\kr(A^2)\subseteq\cdots\quad\text{and}\quad$$
$$\ra( A^0)=\vr\supseteq \ra(A)\supseteq \ra(A^2)\supseteq\cdots$$
\begin{definition}\label{HD1}
Let $A\in\B(\vr)$. Then the hyper-range of $A$ is denoted by $A^\infty(\vr)$ and
$$A^\infty(\vr)=\bigcap_{n\in\N}\ra(A^n)$$
and the hyper-kernel of $A$ is denoted by $$N^{\infty}(A)=\bigcup_{n\in\N}\kr(A^n).$$
\end{definition}
%%%%%%%%%%%%%%%
\begin{proposition}\label{HP1}
	Let $A\in\B(\vr)$, then $A^{\infty}(\vr)$ and $N^\infty(A)$ are $A$-invariant right linear subspaces of $\vr$.
\end{proposition}
\begin{proof}
	Proof is elementary.
\end{proof}
%%%%%%%%%%%%%%%%%%%%
\begin{lemma}\label{1.2}
	Let $A\in\B(\vr)$. For $\qu\in\quat$, if $P_1(\qu)$ and $P_2(\qu)$ are co-prime polynomials with real coefficients then there exist polynomials $Q_1(\qu)$ and $Q_2(\qu)$ with real coefficients such that $P_1(A)Q_1(A)+P_2(A)Q_2(A)=\Iop$.
\end{lemma}
\begin{proof}
	Since the polynomials have real coefficients it follows from the classical case. See lemma 1.2 in \cite{Ai}.
\end{proof}
%%%%%%%%%%%%
The following results establish some basis properties of hype-kernels and hyper-ranges which will be needed in the sequel. 
\begin{theorem}\label{1.3}
	Let $A\in\B(\vr)$. Then
	\begin{enumerate}
		\item [(a)] $R_\qu(A)(N^\infty(\vr))=N^\infty(\vr)$ for every $0\not=\qu\in\quat$;
		\item[(b)]$N^\infty(R_\qu(A))\subseteq (A^2)^\infty(\vr)$  for every $0\not=\qu\in\quat$.
	\end{enumerate}
\end{theorem}
\begin{proof}
	(a)~~In order to prove (a) we need to show that $R_\qu(\kr(A^n))=\kr(A^n)$ for all $n\in\N$ and $\qu\not=0$. Clearly $R_\qu(A)(\kr(A^n))\subseteq \kr(A^n)$ for all $n\in\N$. Since, for $\qu\not=0$, $R_\qu(\pu)$ and $\pu^n$ are co-prime polynomials with real coefficients. Therefore, by lemma \ref{1.2}, there are polynomials $Q_1(\pu)$ and $Q_2(\pu)$ with real coefficients such that
	$$R_\qu(A)Q_1(A)+A^nQ_2(A)=\Iop.$$
	If $\phi\in\kr(A^n)$, then $R_\qu(A)Q_1(A)\phi=\phi$, and since, as $A^n$ and $Q_1(A)$ commute, $Q_1(A)\phi\in\kr(A^n)$. Therefore, $\phi\in R_\qu(A)(\kr(A^n))$, and hence $\kr(A^n)\subseteq R_\qu(A)(\kr(A^n))$. That is, $R_\qu(A)(\kr(A^n))=\kr(A^n)$ for all $n\in\N$ and $\qu\not=0$.\\
	(b)~First we prove that $\kr(R_\qu(A)^n)=A^2(\kr(R_\qu(A)^n))$ for all $n\in\N$ and $\qu\not=0$. clearly $A^2(\kr(R_\qu(A)^n))\subseteq \kr(R_\qu(A)^n)$ for all $n\in\N$. Since, for $\qu\not=0$ and for any $n\in\N$, $\pu^2$ and $R_\qu(\pu)^n$ are co-prime polynomials with real coefficients, there exist polynomials $P(\pu)$ and $Q(\pu)$ with real coefficients such that $A^2P(A)+Q(A)R_\qu(A)^n=\Iop$ for all $n\in\N$ and $\qu\not=0$. Therefore, by the same argument of part (a), we have  $\kr(R_\qu(A)^n)=A^2(\kr(R_\qu(A)^n))$ for all $n\in\N$ and $\qu\not=0$. Hence $N^\infty(R_\qu(A))=A^2(N^\infty(R_\qu(A)))$ for all $\qu\not=0$. From this it easily follows that  $N^\infty(R_\qu(A))=(A^2)^n(N^\infty(R_\qu(A)))$ for all $\qu\not=0$ and $n\in\N$. Therefore, $N^\infty(R_\qu(A))\subseteq(A^2)^\infty(N^\infty(R_\qu(A)))$ for all $\qu\not=0$.
\end{proof}
%%%%%%%%%%%%
%%%%%%%%%%%%%%%%%
\begin{proposition}\label{HP2}
Let $A\in\B(\vr)$ then $A^m(\kr(A^{m+n}))=\ra(A^m)\cap\kr(A^n)$ for all $m,n\in\N$.
\end{proposition}
\begin{proof}
A proof follows exactly as a complex proof. For a complex proof see lemma 1.4 in \cite{Ai}.
\end{proof}
%%%%%%%%%%%%%%%%%%%%%%
\begin{theorem}\label{HT1}
Let $A\in\B(\vr)$. The following statements are equivalent.
\begin{enumerate}
	\item[(a)] $\kr(A)\subseteq A^m(\vr)$ for all $m\in\N.$
	\item[(b)]$\kr(A^n)\subseteq A(\vr)$ for each $n\in\N$.
	\item[(c)] $\kr(A^n)\subseteq A^m(\vr)$ for each $n\in\N$ and each $m\in\N$.
	\item[(d)] $\kr(A^n)=A^m(\kr(A^{m+n}))$ for each $n\in\N$ and each $m\in\N$.
\end{enumerate}
\end{theorem}
\begin{proof}
A proof follows exactly as a complex proof. For a complex proof see Theorem 1.5 in \cite{Ai}.	
\end{proof}
%%%%%%%%%%%%%%%%%%%%
\begin{corollary}\label{HC1}
Let $A\in\B(\vr)$. Then the statements of theorem \ref{HT1} are equivalent to each of the following inclusions.
\begin{enumerate}
	\item [(i)]$\kr(A)\subseteq A^\infty(\vr)$.
	\item[(ii)] $N^\infty(A)\subseteq A(\vr).$
	\item[(iii)] $N^\infty(A)\subseteq A^\infty(\vr).$
\end{enumerate}
\end{corollary}
\begin{proof}
Straightforward from the statements of theorem \ref{HT1}.
\end{proof}
\subsection{Algebraic core of a right linear operator}
\begin{definition}\label{HD2}
Let $A\in\B(\vr)$. The algebraic core, $C(A)$, is defined to be the greatest subspace $M$ of $\vr$ for which $A(M)=M$.
\end{definition}
\begin{remark}\label{HR1}~~
\begin{enumerate}
\item[(a)] Clearly if $A\in\B(\vr)$ is surjective, then $C(A)=\vr$.
\item[(b)] Let $A\in\B(\vr)$, then clearly $C(A)=A^n(C(A))\subseteq A^n(\vr)$ for all $n\in\N$. Thus $C(A)\subseteq\bigcap_{n\in\N}A^n(\vr)=A^\infty(\vr)$.
\end{enumerate}
\end{remark}
\begin{theorem}\label{HT2}
Let $A\in\B(\vr)$ and\\
$$M=\{\phi\in\vr~|~\exists~~\{\psi_n\}_{n=0}^\infty\subseteq\vr~\text{such that}~\phi=\psi_0
~~\text{and}~~A\psi_{n+1}=\psi_n,~~\forall~n\in\Z_+\}.
$$ 
Then $C(A)=M$.
\end{theorem}
\begin{proof}
A proof follows exactly as a complex proof. For a complex proof see Theorem 1.8 in \cite{Ai}.	
\end{proof}
\begin{proposition}\label{HP5}
Let $A\in\B(\vr)$. Suppose there exists $m\in\N$ such that $\kr(A)\cap A^m(\vr)=\kr(A)\cap A^{m+k}(\vr)$ for all $k\geq 0$, then $C(A)=A^\infty(\vr)$.
\end{proposition}
\begin{proof}
	A proof follows exactly as a complex proof. For a complex proof see Lemma 1.9 in \cite{Ai}.	
\end{proof}
\begin{theorem}\label{HT5}
	Let $A\in\B(\vr)$. Suppose that one of the following conditions holds:
	\begin{enumerate}
		\item [(a)]$\dim(\kr(A))<\infty$
		\item[(b)] $\text{codim}(A(\vr))<\infty$
		\item[(c)] $\kr(A)\subseteq A^n(\vr)$ for all $n\in\N$.		
	\end{enumerate}
Then $C(A)=A^\infty(\vr)$.
\end{theorem}
\begin{proof}
	A proof follows exactly as a complex proof. For a complex proof see Theorem 1.10 in \cite{Ai}.	
\end{proof}
%%%%%%%%%%%%%%%%
\subsection{Semi-regular operators on $\vr$}
In the complex theory, the semi-regular operators play an important role in the definition of Kato spectrum and for this reason the Kato spectrum is sometimes referred to as semi-regular spectrum. The same argument applies to the S-spectrum.
\begin{definition}\label{HD3}
Let $A\in\B(\vr)$. $A$ is said to be semi-regular if $\ra(A)$ is closed and $A$ verifies one of the equivalent conditions of theorem \ref{HT1}.
\end{definition}
\begin{example}~~
\begin{enumerate}
	\item[(a)] If $A\in\B(\vr)$ is surjective, then clearly $A$ is semi-regular.
	\item[(b)] If $A\in\B(\vr)$ is injective with closed range, then $A$ is semi-regular.
\end{enumerate}	
\end{example}
%%%%%%%%%%%%%%%%%%%
A semi-regular operator has closed range. So it is useful to find conditions which ensures that $A(\vr)$ is closed. In this regard, the the following quantity associated with $A$ is useful.
\begin{definition}
If $A\in\B(\vr,U_\quat^R)$, the reduced minimum modulus of a nonzero operator $A$ is defined to be
$$\gamma(A)=\inf_{\phi\notin\kr(A)}\frac{\|A\phi\|}{\text{dist}(\phi,\kr(A))}.$$
If $A=0$, then we take $\gamma(A)=\infty$.
\end{definition}
\begin{proposition}\label{HP3}
Let $A\in\B(\vr)$. 
\begin{enumerate}
	\item [(a)] If $A$ is invertible, then $\gamma(A)=\|A^{-1}\|^{-1}$.
	\item[(b)] $\gamma(A)=\gamma(A^\dagger).$
\end{enumerate}
\end{proposition}
\begin{proof}
A proof follows exactly as a complex proof. For details see \cite{La}, page 203.
\end{proof}
\begin{theorem}\label{HT3}
Let $A\in\B(\vr)$. Then
$\gamma(A)>0$ if and only if $\ra(A)$ is closed.
\end{theorem}
\begin{proof}
A proof follows exactly as a complex proof. For a complex proof see Theorem 1.13 in \cite{Ai}.
\end{proof}	
\begin{proposition}\label{HP4}
If $A\in\B(\vr)$ is bounded below then $A$ is semi-regular.
\end{proposition}
\begin{proof}
Proof is elementary. 
\end{proof}
\begin{theorem}\label{HT4}
	Let $A\in\B(\vr)$ is semi-regular, then
	\begin{enumerate}
		\item [(a)] $\gamma(A^n)\geq\gamma(A)^n$ for all $n\in\N$.
		\item[(b)] $A^n$ is semi-regular for all $n\in\N$.
	\end{enumerate}
\end{theorem}
\begin{proof}
	A proof follows exactly as a complex proof. For a complex proof see Theorem 1.16 and corollary 1.17 in \cite{Ai}.
\end{proof}	
%%%%%%%%%%%%%%%%%%%%%
\subsection{Analytical core of $A\in\B(\vr)$:}
In some sense, the analytic core is the analytic counterpart of $C(A)$ \cite{Ai}.
\begin{definition}\label{AD1}\cite{Ai}
Let $A\in\B(\vr)$. The analytical core of $A$ is the set $K(A)$ of all $\phi\in\vr$ such that there exists a sequence $\{u_n\}_{n=0}^\infty\subseteq\vr$ and a constant $\delta>0$ such that
\begin{enumerate}
	\item [(i)]$\phi=u_0$ and $Au_{n+1}=u_n$ for all $n\in\Z_+$.
	\item[(ii)] $\|u_n\|\leq\delta^n\|\phi\|$ for all $n\in\Z_+$.
\end{enumerate}
\end{definition}
\begin{theorem}\label{AT1}
	Let $A\in\B(\vr)$. then
	\begin{enumerate}
		\item [(a)] $K(A)$ is a right linear subspace of $\vr$;
		\item[(b)] $A(K(A))=K(A)$;
		\item[(c)] $K(A)\subseteq C(A)$.
	\end{enumerate}
\end{theorem}
\begin{proof}
$\phi\qu\in K(A)$ for each $\phi\in K(A)$ and $\qu\in\quat$ is straightforward. The rest follows a complex proof. For a complex proof see Theorem 1.21 in \cite{Ai}.
\end{proof}
\begin{theorem}\label{AT2}
	Let $A\in\B(\vr)$.
	\begin{enumerate}
		\item [(a)] If $F$ is a closed subspace of $\vr$ such that $A(F)=F$, then $F\subseteq K(A)$.
		\item[(b)] If $C(A)$ is closed, then $C(A)=K(A)$.
	\end{enumerate}
\end{theorem}
\begin{proof}
	A proof follows exactly as a complex proof. For a complex proof see Theorem 1.22 in \cite{Ai}.
	\end{proof}
\begin{theorem}\label{AT3}
	Let $A\in\B(\vr)$ be a semi-regular operator. If $\phi\in\vr$, then $A\phi\in C(A)$ if and only if $\phi\in C(A)$.
	\end{theorem}	
\begin{proof}
	A proof follows exactly as a complex proof. For a complex proof see Theorem 1.23 in \cite{Ai}.
\end{proof}
\begin{theorem}\label{AT4}
	Let $A\in\B(\vr)$ be a semi-regular operator. Then $C(A)$ is closed and $C(A)=K(A)=A^\infty(\vr)$.
\end{theorem}	
\begin{proof}
	A proof follows exactly as a complex proof. For a complex proof see Theorem 1.24 in \cite{Ai}.
\end{proof}
%%%%%%%%%%%%%%%%%%
\section{local S-spectrum on $\vr$}
\begin{definition}\label{SVEP}\cite{La}
	An operator $A\in\B(\vr)$ has the single-valued extension properly, abbreviated SVEP, at $\q_0\in\quat$ if for every open neighborhood $U\subseteq\quat$ of $\qu_0$, the only continuous right slice-regular solution $f:U\longrightarrow\vr$ of the equation $R_\qu(A)f(\qu)=0$ for all $\qu\in U$ is the zero function on $U$. The operator $A$ is said to have the SVEP if $A$ has the SVEP at every point $\qu\in\quat$.
	\end{definition}
\begin{definition}\cite{La}\label{LS}
	Let $A\in\B(\vr)$ the local S-resolvent set $\rho_A^S(\phi)$ of $A$ at a point $\phi\in\vr$ is defined as the union of all open subsets $U$ of $\quat$ for which there is a continuous  right slice-regular function $f:U\longrightarrow\vr$ which satisfies
	$$R_\qu(A)f(\qu)=\phi,\quad\text{for all}~~\qu\in U.$$
	The local S-spectrum $\ls(\phi)$ of $A$ at $\phi$ is then defined as
	$$\ls(\phi)=\quat\setminus\lr(\phi).$$
	\end{definition}
\begin{remark}\label{LR1} Let $A\in\B(\vr)$ and $\phi\in\vr$. Then
	\begin{enumerate}
		\item[(a)] Since $\lr(\phi)$ is the union of open sets, it is an open set in $\quat$, and hence $\ls(\phi)$ is a closed set in $\quat$.
		\item[(b)] Let $\phi\not=0$ and $\qu\in\rho_S(A)$, then $\kr(\R_\qu(A))=\{0\}$. We have the right inverse $R_{\qu}^{-1}(A):\ra(A^2)\longrightarrow\vr$ and it is right-slice regular in $\qu$ \cite{Jo}. Let $U\subseteq\quat$ is open and define $f:U\longrightarrow\vr$ by $f(\qu)=R_\qu^{-1}(A)\phi$ for all $\qu\in U$, then $R_\qu(A)f(\qu)=\phi$ for all $\qu\in U$. Hence $\qu\in\lr(\phi)$. That is 
		\begin{equation}\label{LE1}
		\rho_S(A)\subseteq\lr(\phi),\quad\text{and hence}\quad\ls(\phi)\subseteq\sigma_S(A).
		\end{equation}
	\end{enumerate}
\end{remark}
\begin{definition}\cite{La}\label{LD1} Let $A\in\B(\vr)$ and $F\subseteq\quat$. The local  S-spectral subspace of $A$ associated with $F$ is defined by
	$$X_A(F)=\{\phi\in\vr~~|~~\ls(\phi)\subseteq F\}.$$
	\end{definition}
\begin{definition}\label{LD2}\cite{La,Ai} Let $A\in\B(\vr)$ and $F\subseteq\quat$ be a closed subset. The set $\mathcal{X}_A(F)$ consists of all $\phi\in\vr$ for which there exists a right slice-regular function $f:\quat\setminus F\longrightarrow\vr$ that satisfies $R_\qu(A)f(\qu)=\phi$ for all $\qu\in\quat\setminus F$. The set $\mathcal{X}_A(F)$ is called the global S-spectral subset of $A$ associated with the set $F$.
	\end{definition}
%%%%%%%%%%%%%%%%%%%
The following proposition shows that, among other results, the surjectivity S-spectrum is closely related to the local S-spectrum.
\begin{proposition}\label{LP1}
Let $A\in\B(\vr)$. Then,
\begin{enumerate}
	\item [(a)] for every $\pu\in\quat\setminus\sus(A)$, there is an $r>0$ for which $\vr=\mathcal{X}_A(\quat\setminus B_\quat(\pu,r))$;
	\item[(b)] $\sus(A)=\bigcup\{\ls(\phi)~|~\phi\in\vr\}$;
	\item[(c)]if $A$ has SVEP and $\qu\in\sigma_{pS}(A)$, then $\ls(\phi)=\{\qu,\oqu\}$ for each eigenvector $\phi$ of $A$ with respect to $\qu$;
	\item[(d)] $\sigma_S(A)=\sus(A)$ if $A$ has SVEP, and $\sigma_S(A)=\apo(A)$ if $A^\dagger$ has SVEP.
\end{enumerate}	
\end{proposition}
\begin{proof}
(a)~ By an obvious translation argument, it is suffices to consider the case where $\pu=0$. Thus $R_{\pu}(A)=R_0(A)=A^2$. Since $\pu=0\in\quat\setminus\sus(A)$, $A^2$ is surjective, and hence $A$ is surjective. Then by the open mapping theorem, there exists $c>0$ such that for every $u\in\vr$, there is some $v\in\vr$ such that $Av=u$ and $c\|v\|\leq\|u\|$. Let $\phi\in\vr$ be arbitrary. Starting with $\phi_0=\phi$ we obtain, by induction, a sequence $\{\phi_n\}\subseteq\vr$ such that $A\phi_n=\phi_{n-1}$ and $c\|\phi_n\|\leq\|\phi_{n-1}\|$, for all $n\in\N$. Therefore, since $\|\phi_n\|\leq c^{-n}\|\phi\|$, we conclude that, for any fixed $\qu\in B_\quat(0,c)$, the series
$$\psi_\qu=\sum_{n=0}^\infty\phi_{n+1}\oqu^n$$
converges locally uniformly. If we do the same for the vector $\psi_\qu\in\vr$, we can obtain another sequence $\{\psi_n\}\subseteq\vr$ such that $\psi_0=\psi_\qu$, $A\psi_n=\psi_{n-1}$ with $d\|\psi_n\|\leq\|\psi_{n-1}\|$ for all $n\in\N$. Define
$$f(\qu)=\sum_{n=0}^\infty\psi_{n+1}{\qu}^{n},$$
which converges locally uniformly on the open ball $B_\quat(0,r)$, where $r=\min\{c,d\}$, and hence $f$ is right-slice regular in $\qu$. We have
$$(A-\qu\IV)f(\qu)=\sum_{n=0}^\infty\psi_n\qu^n-\sum_{n=0}^\infty\psi_{n+1}\qu^{n+1}=\psi_0=\psi_\qu.$$
Therefore
\begin{eqnarray*}
R_\qu(A)f(\qu)&=&(A-\oqu\IV)(A-\qu\IV)f(\qu)\\
&=&(A-\oqu\IV)\psi=(A-\oqu\IV)\sum_{n=0}^\infty\phi_{n+1}\oqu^n\\
&=&\sum_{n=0}^\infty\phi_n\oqu^n-\sum_{n=0}^\infty\phi_{n+1}\oqu^{n+1}=\phi_0=\phi.
\end{eqnarray*}
That is, $R_\qu(A)f(\qu)=\phi$ for all $\qu\in B_\quat(0,r)$, hence $\phi\in\mathcal{X}(\quat\setminus B_\quat(0,r))$, and therefore $\vr=\mathcal{X}(\quat\setminus B_\quat(0,r)).$\\
(b)~For arbitrary $\qu\in\quat$, to prove the equality in (b) it is enough to show that $R_\qu(A)$ is surjective if and only if $\qu\in\lr(\phi)$ for every $\phi\in\vr$. Suppose $\qu\in\lr(\phi)$ and $\phi\in\vr$. Then there is a right regular function on a neighborhood of $U$ of $\qu$, $f: U\longrightarrow\vr$ such that $R_\pu(A)f(\pu)=\phi$ for all $\pu\in U$. Thus $R_\qu(A)$ is surjective. Conversely suppose that $R_\qu(A)$ is surjective Then $\qu\in\quat\setminus\sus(A)$. Therefore, from part (a) $R_\qu(A)(\vr)=\vr=\mathcal{X}_A(\quat\setminus B_\quat(\qu,r))$. Therefore, there is a right-slice regular function $f:B_\quat(\qu,r)\longrightarrow\vr$ such that, for every $\phi\in\vr$, $R_\qu(A)f(\qu)=\phi$ for all $\qu\in B_\quat(\qu,r)$. Hence $\qu\in\lr(\phi)$.\\
(c)~Suppose that $\qu\in\sigma_{pS}(A)$. Then there is a nonzero $\phi\in\vr$ such that $R_\qu(A)\phi=0$. Since the right eigenvalues coincide with the point spectrum, we also have $A\phi=\phi\qu$. Define $f:\quat\setminus\{\qu,\oqu\}\longrightarrow\vr$ by
$$f(\pu)=\phi(\qu^2-2\text{Re}(\pu)\qu+|\pu|^2)^{-1},\quad\text{for all}~~\pu\in\quat\setminus\{\qu,\oqu\}.$$
Then $f$ is right-slice regular on $\quat\setminus\{\qu,\oqu\}$ (see \cite{Jo}, page 81) and satisfies, as $A\phi=\phi\qu$,
\begin{eqnarray*}
	R_\pu(A)f(\pu)&=&(A^2-2\text{Re}(\pu)A-|\pu|^2)\phi(\qu^2-2\text{Re}(\pu)\qu+|\pu|^2)^{-1}\\
	&=&\phi(\qu^2-2\text{Re}(\pu)\qu+|\pu|^2)(\qu^2-2\text{Re}(\pu)\qu+|\pu|^2)^{-1}=\phi\quad\text{for all}~\pu\in\quat\setminus\{\qu\}.
	\end{eqnarray*}
Therefore $\quat\setminus\{\qu,\oqu\}\subseteq\lr(\phi)$, and hence $\ls(\phi)\subseteq\{\qu,\oqu\}$. For the other inclusion, assume that $\qu\in\lr(\phi)$. Then there exists a right-slice regular function $f:U\longrightarrow\vr$ on some open neighborhood of $\qu$ such that $R_\pu(A)f(\pu)=\phi$ for all $\pu\in U$. It follows that
$$R_\pu(A)R_\qu(A)f(\pu)=R_\qu(A)R_\pu(A)f(\pu)=R_\qu(A)\phi=0,\quad\text{for all}~\pu\in U.$$ Therefore by SVEP, $R_\qu(A)f(\pu)=0$, for all $\pu\in U$. In particular, $0\not=\phi=R_\qu(A)f(\qu)=0$, which is a contradiction. Hence, $\qu\in\ls(\phi)$.\\
(d)~~Suppose $A$ has SVEP, then by parts (b) and (c), we have
$$\sigma_{pS}(A)\subseteq\bigcup\{\ls(\phi)~|~\phi\in\vr\}=\sus(A).$$
Therefore by equation \ref{sue1}, we get 
\begin{equation}\label{LE2}\sigma_S(A)=\sus(A).\end{equation}
If $A^\dagger$ has SVEP, then by equation \ref{LE2} and proposition \ref{su1}, we have
$$\sigma_S(A)=\sigma_S(A^\dagger)=\sus(A^\dagger)=\apo(A).$$
\end{proof}
%%%%%%%%%%%%%%%%%%%%
The following proposition relates isolated points of various spectra to SVEP.
\begin{proposition}\label{SV2}
Let $A\in\B(\vr)$.
\begin{enumerate}
	\item [(a)]  If $\sigma_{pS}(A)$ does not cluster at $\qu_0\in\quat$, then $A$ has SVEP at $\qu_0$.
	\item[(b)] If $\apo(A)$ does not cluster at  $\qu_0\in\quat$, then $A$ has SVEP at $\qu_0$.
	\item[(c)] If $\sus(A)$ does not cluster at  $\qu_0\in\quat$, then $A^\dagger$ has SVEP at $\qu_0$.
\end{enumerate}
\end{proposition}
\begin{proof}
(a)~Suppose that $\sigma_{pS}(A)$ does not cluster at $\qu_0$. Then there exists a neighborhood $U$ of $\qu_0$ such that $R_\qu(A)$ is injective for all $\qu\in U$ and $\qu\not=\qu_0$. Let $f:V\longrightarrow\vr$ be a right-slice regular function defined on another neighborhood of $\qu_0$ for which the equation $R_\qu(A)f(\qu)=0$ holds for every $\qu\in V$. Obviously we may assume that $V\subseteq U$. Then $f(\qu)\in\kr(R_\qu(A))=\{0\}$ for all $\qu\in V$ and $\qu\not=\qu_0$. Hence $f(\qu)=0$  for all $\qu\in V$ and $\qu\not=\qu_0$. From the continuity of $f$ at $\qu_0$ we conclude that $f(\qu_0)=0$. Hence $f=0$ on $V$, and therefore $A$ has SVEP at $\qu_0$.\\
(b)~Suppose that $\apo(A)$ does not cluster at $\qu_0$. Then there is a neighborhood $U$ of $\qu_0$ such that $U\setminus\{\qu_0\}\cap\apo(A)=\emptyset$. Since, by proposition \ref{P3}, $\sigma_{pS}(A)\subseteq\apo(A)$, we have $U\setminus\{\qu_0\}\cap\sigma_{pS}(A)=\emptyset$. Therefore, from the proof of part (a), $A$ has SVEP.\\
(c)~Since, by proposition \ref{su1}, $\sus(A)=\apo(A^\dagger)$. Therefore, from part (b) $A^\dagger$ has SVEP.
\end{proof}
%%%%%%%%%%%
\begin{remark}
	\begin{enumerate}
		\item [(a)] From proposition \ref{SV2} every operator $A\in\B(\vr)$ has SVEP at an isolated point of the $S$-spectrum.
		\item[(b)] Obviously $A\in\B(\vr)$ has SVEP at every $\qu\in\rho_S(A)$.
	\end{enumerate}
\end{remark}
%%%%%%%%%%%%%%
The following proposition gathers some elementary properties of local spectral subspaces. 
\begin{proposition}\label{1.2.16}
For every operator $A\in\B(\vr)$ and every set $F\subseteq\quat$, the following assertions hold:
\begin{enumerate}
	\item [(a)] $X_A(F)$ is an $A$-hyper-invariant right linear subspace of $\vr$;
	\item[(b)] $R_\qu(A)X_A(F)\subseteq X_A(F)$ for all $\qu\in\quat\setminus F$;
	\item[(c)] if $Y$ is an $A$-invariant closed right linear subspace of $\vr$ with the property that $\sigma_S(A|_Y)\subseteq F$, then $Y\subseteq X_A(F)$;
	\item[(d)] $X_A(F)=X_A(F\cap\sigma_S(A))$.
\end{enumerate}
\end{proposition}
\begin{proof}
(a)~We have the right-slice regular function $f:\quat\longrightarrow\vr$ defined by $f(\qu)=0$ for all $\qu\in\quat$ such that $R_\qu(A)f(\qu)=0$ for all $\qu\in\quat$. Thus $\lr(0)=\quat$ and hence $\ls(0)=\emptyset\subseteq F$. Therefore $0\in X_A(F)$. Let $\phi, \psi\in X_A(F)$. If $\qu\in\lr(\phi)\cap\lr(\psi)$, then there are slice-regular functions $f$ and $g$ on some open neighborhood $U$ and $V$ of $\qu$, $f:U\longrightarrow\vr$ and $g:V\longrightarrow\vr$, such that $R_\pu(A)f(\pu)=\phi$ for all $\pu\in U$ and $R_\pu(A)g(\pu)=\psi$ for all $\qu\in V$. Hence $f+g:U\cap V\longrightarrow\vr$, a right-slice regular function, such that $R_\pu(A)(f+g)(\pu)=\phi+\psi$ for all $\pu\in U\cap V$. Thus $\qu\in\lr(\phi+\psi)$ and hence $\lr(\phi)\cap\lr(\psi)\subseteq\lr(\psi+\phi)$. Therefore $\ls(\phi+\psi)\subseteq\ls(\phi)\cup\ls(\psi)\subseteq F$, which implies $\phi+\psi\in X_A(F)$. Let $s\in\quat\setminus\{0\}$ and $\qu\in\lr(\phi s)$, then there is an open neighborhood $U$ of $\qu$ and a right-slice regular function $f:U\longrightarrow\vr$ such that $R_\pu(A)f(\pu)=\phi s$ for all $\pu\in U$. Then $R_\pu(A)f(\pu)s^{-1}=\phi$ for all $\pu\in U$ and $f(\pu)s^{-1}$ is right-slice regular in the variable $\pu$, thus $\qu\in\lr(\phi)$. Therefore $\lr(\phi s)\subseteq\lr(\phi)$ and $\ls(\phi)\subseteq\ls(\phi s)$. If $\qu\in\lr(\phi)$, then there exist an open neighborhood $U$ of $\qu$ and a right-slice regular function $f:U\longrightarrow\vr$ such that $R_\pu(A)f(\pu)=\phi$ for all $\pu\in U$. Define the function $sf:U\longrightarrow\vr$ by $(sf)(\pu)=f(\pu)s$. Then $sf$ is right-slice regular on $U$ and $R_\pu(A)(sf)(\qu)=\phi s$. Hence $\qu\in\lr(\phi s)$. Therefore $\lr(\phi)\subseteq\lr(\phi s)$ and $\ls(\phi s)\subseteq\ls(\phi)$. Thus we have $\ls(\phi)=\ls(\phi s)\subseteq F$. That is $\phi s\in X_A(F)$. Therefore $X_A(F)$ is a right linear subspace of $\vr$.\\
For hyper-invariance, let $B\in\B(\vr)$ commute with $A$. Let $\phi\in X_A(F)$. If $\qu\not\in\ls(\phi)$ then there is an open neighborhood $U$ of $\qu$ and a right-slice regular function $f:U\longrightarrow\vr$ such that $R_\pu(A)f(\pu)=\phi$ for all $\pu\in U$. Now $B\circ f:U\longrightarrow\vr$ is right-slice regular and $R_\pu(A)(S\circ f)(\pu)=SR_\pu(A)f(\pu)=S\phi$ for all $\pu\in U$. Therefore $\qu\in\lr(S\phi)$ and $\qu\not\in\ls(S\phi)$. Hence $\ls(S\phi)\subseteq\ls(\phi)\subseteq F$. That is, $S\phi\in X_A(F)$, and therefore $S(X_A(F))\subseteq X_A(F)$.\\
(b)~~Since $R_\qu(A)$ and $A$ commutes, it is straight forward from part(a).\\
(c)~~Let $\phi\in Y$, then we have $R_\qu(A)R_\qu(A|_Y)^{-1}\phi=\phi$ for all $\qu\in\rho_S(A|_Y)$, where $A|_Y\in\B(Y)=\{A~:~A:Y\longrightarrow Y\text{~~is bounded right}~~ \quat-\text{linear operator}\}$, and hence  $R_\qu(A)R_\qu(A|_Y)^{-1}\phi=\phi$ for all $\qu\in\quat\setminus F$. That is, $f:\quat\setminus F\longrightarrow\vr$ by $f(\qu)=R_\qu(A|_Y)^{-1}\phi$ is a right-slice regular function such that $R_\qu(A)f(\qu)=\phi$ for all $\qu\in\quat\setminus F$. Therefore, $\qu\in\lr(\phi)$ for all $\qu\in\quat\setminus F$, and hence $\ls(\phi)\subseteq F$. Thus $\phi\in X_A(F)$, and which yields $Y\subseteq X_A(F)$.\\
(d)~~Clearly $X_A(F\cap\sigma_S(A))\subseteq X_A(F)$. Conversely, let $\phi\in X_A(F)$, then $\ls(\phi)\subseteq F$, and since $\ls(\phi)\subseteq\sigma_S(A)$, we get $\ls(\phi)\subseteq F\cap\sigma_S(A)$. Thus $\phi\in X_A(F\cap\sigma_S(A))$.
\end{proof}
%%%%%%%%%%%%%%%%%%%%%%%%%%
\section{Kato S-spectrum in $\vr$}
In the complex setting, among the many concepts dealt with in Kato's extensive treatment of perturbation theory \cite{Kato} there is a very important part of the spectrum called the Kato spectrum. Here we duplicate the complex definition given in \cite{Ai,La} to quaternions.
\begin{definition}\label{DK1}
For $A\in\B(\vr)$, the Kato S-resolvent set is defined as
$$\rho_{ka}^S(A)=\{\qu\in\quat~~|~~\ra(R_\qu(A))~\text{is closed}~~ \text{and}~\kr(R_\qu(A))\subseteq R_\qu(A)^\infty(\vr)\}$$
and the Kato S-spectrum is defined as $\sigma_{ka}^S(A)=\quat\setminus\rka(A)$.
\end{definition}
%%%%%%%%%%%%%%%%%
\begin{remark}\label{KR1}Let $A\in\B(\vr)$.
\begin{enumerate}
	\item [(a)] From theorem \ref{HT1} we can see that:\\
	$\qu\in\rka(A)$ if and only if $\ra(R_\qu(A))$ is closed and $R_\qu(A)$ satisfies one of the equivalent conditions of theorem \ref{HT1}.
	That is
	$$\rka(A)=\{\qu\in\quat~|~R_\qu(A)~~\text{is semi-regular}\}.$$
	\item[(b)] In the complex literature the Kato spectrum is sometimes referred to as semi-regular spectrum. For example in \cite{La} it is called Kato spectrum while in \cite{Ai} it is referred as semi-regular spectrum.
	\item[(c)] Let $\qu\in\rho_S(A)$, then $R_\qu(A)$ has an inverse in $\B(\vr)$. Therefore, by the bounded inverse theorem, $R_\qu(A)$ is bounded below, and hence by proposition \ref{HP4}, $R_\qu(A)$ is semi-regular. Thus $\qu\in\rka(A).$ That is, $\rho_S(A)\subseteq\rka(A)$, and hence $\ska(A)\subseteq\sigma_S(A)$.
\end{enumerate}
\end{remark}
\begin{proposition}\label{KP0}
	Let $A\in\B(\vr)$, then $\quat\setminus\apo(A)\subseteq\rka(A)$.
\end{proposition}
\begin{proof}
	Let $\qu\in\quat\setminus\apo(A)$, then by proposition \ref{P4}, $\kr(R_\qu(A))=\{0\}$ and $\ra(R_\qu(A))$ is closed. Therefore, $\qu\in\rka(A).$
\end{proof}
\begin{remark}\label{KR0}
	Let $\qu=q_0+q_1i+q_2j+q_3k\in\quat$ and $A\in\B(\vr)$. Denote $\beta(A,\qu)=\gamma(A)^2-2|q_0|\gamma(A)-|\qu|^2$ and $\beta(\qu)=|q_0|+\sqrt{2q_0^2+|\qu|^2}$, then we have
	\begin{eqnarray*}
		\beta(A,\qu)>0&\Leftrightarrow&\gamma(A)^2-2|q_0|\gamma(A)-|\qu|^2>0\\
			&\Leftrightarrow& (\gamma(A)-|q_0|)^2>|q_0|^2+|\qu|^2\\
			&\Leftrightarrow&\gamma(A)>|q_0|+\sqrt{|q_0|^2+|\qu|^2}=\beta(\qu).
		\end{eqnarray*}
	Also note that $\gamma(A)>\beta(\qu)$ implies $\gamma(A)>|q_0|$ also $\beta(\qu)>0$ if $\qu\not=0$.
\end{remark}
%%%%%%%%%%%%%%%%%%%%%%%%%%%%%%%%%%%%%%%%%%%%%%%%%%%%%%%%%%%%%%%
\begin{proposition}\label{KP1}Let $A\in\B(\vr)$ and $\beta(\qu)$ is as in remark \ref{KR0}, then
	\begin{enumerate}
		\item [(a)] $A$ is surjective (respectively, bounded below) if and only if $A^\dagger$ is bounded below (respectively, surjective).
		\item[(b)] if $A$ is bounded below (respectively, surjective) then $R_\qu(A)$ is bounded below (respectively, surjective) for each $\qu\in\quat$ that satisfies $\gamma(A)>\beta(\qu)$.
	\end{enumerate}
\end{proposition}
\begin{proof}
	(a)~Proof is exactly as a complex proof. For a complex proof see lemma 1.30 (a) in \cite{Ai}.\\
	(b)~Suppose $A$ is bounded below. Thus, $A$ is injective and $\ra(A)$ is closed. Hence, as $A$ is continuous, $A^2$ is injective and $A^2(\vr)=A(A(\vr))$ is closed. Therefore, from theorem \ref{HT3} and theorem \ref{HT4}, $\gamma(A)>0$ and $\gamma(A^2)>0$. Also from the injectivity of $A$ and $A^2$,
	$$\gamma(A)\text{dist}(\phi, \kr(A))=\gamma(A)\|\phi\|\leq\|A\phi\|,\quad\text{for all}~\phi\in\vr\quad\text{and}$$
	$$\gamma(A^2)\text{dist}(\phi, \kr(A^2))=\gamma(A^2)\|\phi\|\leq\|A^2\phi\|,\quad\text{for all}~\phi\in D(A^2).$$
	We have, for $\phi\in D(A^2)$,
	\begin{eqnarray*}
		\|R_\qu(A)\phi\|&\geq&\|A^2\phi\|-2|\text{Re}(\qu)|\|A\phi\|-|\qu|^2\|\phi\|\\
		&\geq& \gamma(A^2)\|\phi\|-2|\text{Re}(\qu)|\gamma(A)\|\phi\|-|\qu|^2\|\phi\|\\
		&=&(\gamma(A^2)-2|\text{Re}(\qu)|\gamma(A)-|\qu|^2)\|\phi\|\geq\beta(A,\qu)\|\phi\|\quad{\text{by theorem \ref{HT4}}}.
		\end{eqnarray*}
	Hence, if $\gamma(A)>\beta(\qu)$, then, by remark \ref{KR0}, $\beta(A,\qu)>0$. Therefore, $R_\qu(A)$ is bounded below. Trivially, if $A$ is surjective, then $R_\qu(A)$ is surjective.
\end{proof}
%%%%%%%%%%%%%%%%%%%%%%%%%%%%%%%%%
\begin{theorem}\label{KT1}
Let $A\in\B(\vr)$ be semi-regular. Then $R_\qu(A)$ is semi-regular for all $\qu\in\quat$ for which $\gamma(A)>\beta(\qu)$, where $\beta(\qu)$ is as in remark \ref{KR0}. Moreover $\rka(A)$ is open and $\ska(A)$ is compact.
\end{theorem}
\begin{proof}
First we show that $C(A)\subseteq C(R_\qu(A))$ for all $\qu\in\quat$ with $\gamma(A)>\beta(\qu)$. Let $A_0:C(A)\longrightarrow C(A)$ denote the restriction of $A$ to $C(A)$. Since $A$ is semi-regular, by theorem \ref{AT4}, $C(A)$ is closed. Since $A(C(A))=C(A)$, $A_0$ is surjective. Therefore, by proposition \ref{KP1}, $R_\qu(A_0)$ is surjective for all $\qu\in\quat$ with $\gamma(A_0)>\beta(\qu)$. Thus $R_\qu(A_0)(C(A))=R_\qu(A)(C(A))=C(A)$ for all $\qu\in\quat$ with $\gamma(A_0)>\beta(\qu)$. On the other hand, $A$ is semi-regular, therefore by theorem \ref{HT1}, corollary \ref{HC1} and theorem \ref{HT5}, we have $\kr(A)\subseteq A^\infty(\vr)=C(A)$. This implies, also by theorem \ref{HT3}, $\gamma(A_0)\geq\gamma(A)>0$, 
 \begin{equation}\label{KE1}
 C(A)\subseteq C(R_\qu(A))~~\text{for all}~~\qu\in\quat~~\text{with}~~ \gamma(A)>\beta(\qu).
 \end{equation}
 Moreover, for every $\qu\in\quat\setminus\{0\}$ we have $A(\kr(R_\qu(A)))=\kr(R_\qu(A))$ and $\kr(R_\qu(A))$ is closed, therefore, from theorem \ref{AT2} and theorem \ref{AT4}, we have $\kr(R_\qu(A))\subseteq C(A)$ for all $\qu\in\quat\setminus\{0\}$. We also have $C(R_\qu(A))=R_\qu(A)^n(C(R_\qu(A)))\subseteq R_\qu(A)^n(\vr)$ for all $\qu\in\quat$ and for all $n\in\mathbb{N}$. Therefore, from equation \ref{KE1}, we have, $\text{for each}~~\qu\in\quat\setminus\{0\}$ and for each $n\in\mathbb{N}$,
 \begin{equation}\label{KE2}
 \kr(R_\qu(A))\subseteq C(R_\qu(A))\subseteq R_\qu(A)^n(\vr), ~~\text{with}~~ \gamma(A)>\beta(\qu).
 \end{equation}
Since $A$ is semi-regular, by theorem \ref{HT4}, $A^2$ is semi-regular. Therefore equation \ref{KE2} is valid for $\qu=0$ as well. That is, equation \ref{KE2} is valid for all $\qu\in\quat$ with $ \gamma(A)>\beta(\qu)$.\\
{\em Claim}: $\ra(R_\qu(A))$ is closed for all $\qu\in\quat$ with $\gamma(A)>\beta(\qu)$.\\
If $C(A)=\{0\}$, then as $A$ is semi-regular, by theorems \ref{AT2} and \ref{AT4}, $\kr(A)\subseteq C(A)=\{0\}$. Therefore, by the bounded inverse theorem, $A$ is bounded below, and hence by lemma \ref{KP1} $R_\qu(A)$ is bounded below for all $\qu\in\quat$ with $\gamma(A)>\beta(\qu).$ Thus, by definition, $\ra(R_\qu(A))$ is closed.\\
If $C(A)=\vr$, then $A$ is surjective, therefore, again by proposition \ref{KP1}, so is $R_\qu(A)$.\\
Now consider the case $C(A)\not=\{0\}$ and $C(A)\not=\vr$. Let $\ov=\vr/C(A)$ and let $\overline{A}:\ov\longrightarrow\ov$ be the quotient map defined by $\oa\,\op:=\overline{A\phi}$, where $\op\in\ov$. Clearly $\oa$ is continuous.
If $\oa\,\op=\overline{A\phi}=\overline{0}$, then $A\phi\in C(A)$, thus, by theorem \ref{AT3}, $\phi\in C(A)$ which implies $\op=\overline{0}$. Therefore $\oa$ is injective. Next we prove that $\oa$ is bounded below. To prove it, we only need to show that $\oa$ has closed range. To see this we show the inequality $\gamma(\oa)\geq\gamma(A)$, then, by theorem \ref{HT3}, $\oa$ has closed range. For each $\phi\in\vr$ and each $u\in C(A)$ we have, recalling the fact that $\kr(A)\subseteq C(A)$ and by the definition of the quotient norm,
\begin{eqnarray*}
\|\op\|=\text{dist}(\phi,C(A))&=&\text{dist}(\phi-u, C(A))\\
&\leq&\text{dist}(\phi-u,\kr(A))\leq\frac{1}{\gamma(A)}\|A\phi-Au\|.
\end{eqnarray*}
From the equality $C(A)=A(C(A))$ we obtain that $\displaystyle\|\overline{A\phi}\|=\inf_{u\in C(A)}\|A\phi-Au\|.$ Thus, $\|\op\|\leq\frac{1}{\gamma(A)}\|\overline{A\phi}\|$. That is, $\gamma(A)\leq\frac{\|\overline{A\phi}\|}{\|\op\|}$ for all $\op\in\ov$, from this, as $\oa$ is injective, we get $\gamma(\oa)\geq\gamma(A)$. Hence $\oa$ is bounded below. Therefore, by proposition \ref{KP1}, $R_\qu(\oa)$ is bounded below for all $\qu\in\quat$ with $\gamma(\oa)>\beta(\qu)$ and hence for all $\qu\in\quat$ with $\gamma(A)>\beta(\qu)$. Finally, to show that $\ra(R_\qu(A))$ is closed  for all $\qu\in\quat$ with $\gamma(A)>\beta(\qu)$, let $\{\phi_n\}\subseteq\ra(R_\qu(A))$ be a sequence such that $\phi_n\longrightarrow\phi\in\vr$ as $n\longrightarrow\infty$. Then clearly $\op_n\longrightarrow\op\in\ov$ as $n\longrightarrow\infty$ and $\op_n\in\ra(R_\qu(\oa))$, and this space is closed  for all $\qu\in\quat$ with  $\gamma(A)>\beta(\qu)$, therefore $\op\in\ra(R_\qu(\oa))$. Let $\op=R_\qu(\oa)\overline{v}$ and $v\in\overline{v}\in\ov$. Then $\phi-R_\qu(A)v\in C(A)\subseteq R_\qu(A)(C(A))$  for all $\qu\in\quat$ with $\gamma(A)>\beta(\qu)$. So there exists $u\in C(A)$ such that $\phi=R_\qu(A)(v+u)$, hence $\phi\in\ra(R_\qu(A))$  for all $\qu\in\quat$ with $\gamma(A)>\beta(\qu)$. Therefore, $\ra(R_\qu(A))$ is closed  for all $\qu\in\quat$ with $\gamma(A)>\beta(\qu)$, and, consequently, $R_\qu(A)$ is semi-regular  for all $\qu\in\quat$ with $\gamma(A)>\beta(\qu)$. That is, $\qu\in\rka(A)$  for all $\qu\in\quat$ such that $\gamma(A)>\beta(\qu)$.
 Hence $\qu\in\ska(A)$ if $\qu\in\quat$ satisfies $\beta(\qu)\leq \gamma(A)$. Let $\qu\in\overline{\ska(A)}$, then there exist a sequence $\{\qu_n\}\subseteq\ska(A)$ such that $\qu_n\longrightarrow\qu$ as $n\longrightarrow\infty$. So we have $\beta(\qu_n)\leq \gamma(A)$, hence, as $n\longrightarrow\infty$ we get $\beta(\qu)\leq \gamma(A)$, and therefore $\qu\in\ska(A)$. Thus $\ska(A)$ is closed, consequently, $\rka(A)$ is open. From remark \ref{KR1}, (c) we have $\ska(A)\subseteq\sigma_S(A)$. We know $\sigma_S(A)$ is compact and since a closed subset of a compact set is compact, $\ska(A)$ is compact.
\end{proof}		
%%%%%%%%%%%%%%%
\begin{proposition}\label{3.1.4}
	Suppose that the operator $A\in\B(\vr)$ satisfies $\kr(A)\subseteq A^\infty(\vr)$. Then $A$ maps $A^\infty(\vr)$ onto itself, and $\kr(A^m)\subseteq A^\infty(\vr)$ for all $m\in\N$.
\end{proposition}
\begin{proof}
	A proof follows its complex counterpart. For a complex proof see \cite{La}, lemma 3.1.4.
\end{proof}
\begin{proposition}\label{3.1.3}
	Suppose that the operator $A\in\B(\vr)$ has closed range, and that $Y$ is a closed right linear subspace of $\vr$ that contains $\kr(A)$, then $A(Y)$ is closed.
\end{proposition}
	\begin{proof}
		A proof follows its complex counterpart. For a complex proof see \cite{La}, lemma 3.1.3.
	\end{proof}
%%%%%%%%%%%%%%%%%%%%%%%%%%
\begin{proposition}\label{3.1.5}
Let $A\in\B(\vr)$ and $\qu\in\rka(A)$. Then $R_\qu(A)^m$ has closed range for every $m\in\N$, the space $R_\qu(A)^\infty(\vr)$ is closed, $R_\qu(A)$ maps $R_\qu(A)^\infty(\vr)$ onto itself, and $R_\qu(A)^\infty(\vr)\subseteq X_A(\quat\setminus\{\qu\})$.
\end{proposition}
\begin{proof}
Claim: $R_\qu(A)^m$ has closed range for each $m\in\N$.\\
We prove it by induction. Since $\qu\in\rka(A)$, $\ra(R_\qu(A))$ is closed so the case $m=1$ is clear. Assume that $\ra(R_\qu(A)^m)$ is closed for some $m\geq 1$. Let $Y=\ra(R_\qu(A))$. From proposition \ref{3.1.4}, we know that, as $\qu\in\rka(A)$, $\kr(R_\qu(A)^m)\subseteq R_\qu(A)^\infty(\vr)\subseteq Y$. Therefore, by proposition \ref{3.1.3}, $R_\qu(A)^m(Y)$ is closed. That is, $R_\qu(A)^m(Y)=R_\qu(A)^{m+1}(\vr)$ is closed, which completes the induction.\\
As $R_\qu(A)^m(\vr)$ is closed for all $m\in\N$, their intersection $R_\qu(A)^\infty(\vr)$ is closed. Since $\qu\in\rka(A)$, $\kr(R_\qu(A))\subseteq R_\qu(A)^\infty(\vr)$, therefore, by proposition \ref{3.1.4}, $R_\qu(A)$ maps $R_\qu(A)^\infty(\vr)$ onto itself.
% Moreover, proposition \ref{1.2.16} part (b) implies that $R_\qu(A)(X_A(\quat\setminus\{\qu\}))=X_A(\quat\setminus\{\qu\})$, and consequently, $R_\qu(A)^m(X_A(\quat\setminus\{\qu\}))=X_A(\quat\setminus\{\qu\})$ for all $m\in\N$. Hence $X_A(\quat\setminus\{\qu\})\subseteq R_\qu(A)^\infty(X_A(\quat\setminus\{\qu\})).$
 To prove the  inclusion, we can say from proposition \ref{3.1.4} that the restriction of $R_\qu(A)$ to $R_\qu(A)^\infty(\vr)$ is surjective. Thus $\qu\not\in\sus(B)$, where $B:=A\vert_{R_\qu(A)^\infty(\vr)}$. Let $\phi\in R_\qu(A)^\infty(\vr)$, then part (b) of proposition \ref{LP1} to conclude that
$$\ls(\phi)\subseteq\sigma_B^S(\phi)\subseteq\sus(B)\subseteq\quat\setminus\{\qu\}.$$
This observation shows that $\phi\in X_A(\quat\setminus\{q\}).$ Thus $R_\qu(A)^\infty(\vr)\subseteq  X_A(\quat\setminus\{q\})$.
\end{proof}
%%%%%%%%%%%%%%%%%%%%%%%
Following the complex definition of analytic residuum in \cite{La} we define the following.
\begin{definition}\label{Res}
	Let $A\in\B(\vr)$, the analytic residuum $S(A)$ is the open set of points $\qu\in\quat$ for which there exists a non-vanishing continuous right-slice regular function $f:U\longrightarrow\vr$ on some open neighborhood $U$ of $\qu$ such that $R_\pu(A)f(\pu)=0$ for all $\pu\in U$.
	\end{definition}
\begin{proposition}\label{ResP}
Let $A\in\B(\vr)$, then $S(A)\subseteq\text{int}\sigma_{pS}(A)$, the interior of $\sigma_{pS}(A)$. Moreover $S(A)$ is empty if $A$ has SVEP.
\end{proposition}
\begin{proof}
Let $\qu\in S(A)$, then there exists an open neighborhood $U$ of $\qu$ and a non-vanishing right-slice regular function $f:U\longrightarrow\vr$ such that $R_\pu(A)f(\pu)=0$ for all $\pu\in U$. Since $f(\pu)\not=0$ for all $\pu\in U$, $\kr(R_\pu(A))\not=\{0\}$ for all $\pu\in U$. Hence $\qu\in U\subseteq \sigma_{pS}(A)$. Therefore, $S(A)\subseteq\text{int}\sigma_{pS}(A)$. $S(A)=\emptyset$ if $A$ has SVEP is trivial from the definitions.
\end{proof}
%%%%%%%%%%%%%%%%%%%%%%%%%%%%%%%%%
Proposition \ref{3.1.5} leads to the following sandwich formula for the Kato S-spectrum. In particular, we obtain $\partial\sigma_S(A)\subseteq\ska(A)$, which ensures that $\ska(A)$ is non-empty provided that $\vr$ is non-trivial.
\begin{proposition}\label{3.1.6}
Let $A\in\B(\vr)$, then
\begin{enumerate}
	\item [(a)]$\rka(A)=\rka(A^\dagger)$;
	\item[(b)] $\rka(A)\cap\sigma_S(A)\subseteq S(A)\cup S(A^\dagger)$;
	\item[(c)]$\partial\sigma_S(A)\subseteq(\apo(A)\cap\sus(A))\setminus(S(A)\cap S(A^\dagger)\subseteq\ska(A)\subseteq\apo(A)\cap\sus(A)$;
	\item[(d)]$(\apo(A)\cap\sus(A))\setminus(S(A)\cap S(A^\dagger)=(\apo(A)\setminus S(A))\cup(\sus(A)\setminus S(A^\dagger))$.
\end{enumerate}
\end{proposition}
	 \begin{proof}
	 	(a)~~Let $\qu\in\rka(A)$. Then $\ra(R_\qu(A))$ is closed and by corollary \ref{HC1}, $\kr(R_\qu(A))\subseteq R_\qu(A)^\infty(\vr)$. Then by proposition \ref{3.1.5}, $R_\qu(A)^n$ has closed range for every $n\in\N$, and by proposition \ref{3.1.4}, $\kr(R_\qu(A)^m)\subseteq R_\qu(A)(\vr)$. Hence by the proposition \ref{NP1}, $R_\qu(A^\dagger)^n(\vr)$ is closed and, by proposition \ref{IP30}
	 	$$\kr(R_\qu(A^\dagger))=[R_\qu(A)(\vr)]^\perp\subseteq[\kr(R_\qu(A)^n)]^\perp=R_\qu(A^\dagger)^n(\vr)$$
	 	for all $n\in\N$. Thus $\qu\in\rka(A^\dagger)$, and therefore $\rka(A)\subseteq\rka(A^\dagger)$. The opposite inclusion is similar.\\
	 	(b)~~From proposition \ref{KP0} we have $\quat\setminus\apo(A)\subseteq\rka(A)$. For $\qu\in\quat\setminus\sus(A)$, we have $R_\qu(A)(\vr)=\vr$, and hence trivially $\qu\in\rka(A)$. Therefore, 
	 	\begin{equation}\label{Ke1}
	 	\ska(A)\subseteq\apo(A)\cap\sus(A).
	 	\end{equation}
	 	From theorem \ref{T1} and proposition \ref{su2}, we have $\partial\sigma_S(A)\subseteq\apo(A)\cap\sus(A)$. From proposition \ref{ResP}, we have, 
	 	$$S(A)\cap S(A^\dagger)\subseteq \text{int}\sigma_{pS}(A)\cap\text{int}\sigma_{pS}(A^\dagger)\subseteq\text{int}\sigma_S(A)$$
	 	as $\sigma_S(A)=\sigma_S(A^\dagger)$ by proposition \ref{su1}.
	 	Therefore,
	 	\begin{equation}\label{Ke2}
	 	\partial\sigma_S(A)\subseteq(\apo(A)\cap\sus(A))\setminus(S(A)\cap S(A^\dagger).
	 	\end{equation}
	 	Claim:
	 	\begin{equation}\label{Ke3}
	 	 \rka(A)\cap\apo(A)\subseteq S(A).
	 	\end{equation}
	 	Let $\qu\in\rka(A)\cap\apo(A)$. Since $\qu\in\rka(A)$, $\ra(R_\qu(A))$ is closed, and since $\qu\in\apo(A)$, by proposition \ref{P4}, $\kr(R_\qu(A))\not=\{0\}$. Therefore $\qu$ is a right eigenvalue of $A$. Let $\phi$ be a corresponding eigenvector, then $\phi\in\kr(R_\qu(A))$. Hence, by proposition \ref{3.1.5}, as $\qu\in\rka(A)$,
	 	$$\phi\in\kr(R_\qu(A))\subseteq R_\qu(A)^\infty(\vr)\subseteq X_A(\quat\setminus\{\qu\}).$$
	 	Thus, by the definition of $X_A(\quat\setminus\{\qu\})$, $\qu\in\lr(\phi)$,  there exists a right-slice regular function $f:U\longrightarrow\vr$ on an open neighborhood of $\qu$ for which $R_\pu(A)f(\pu)=\phi$ for all $\pu\in U$. Define the right-slice regular function $g:U\longrightarrow\vr$ by $g(\pu)=R_\qu(A)f(\pu)$ for all $\pu\in U$. Since $R_\qu(A)$ and $R_\pu(A)$ commute, we have
	 	\begin{eqnarray*}
	 	R_\pu(A)g(\pu)&=&	R_\pu(A)R_\qu(A)f(\pu)=R_\qu(A)R_\pu(A)f(\pu)\\
	 	&=&R_\qu(A)\phi=0\quad\text{for all}~~\pu\in U.
	 	\end{eqnarray*}
 	Since $g(\qu)=R_\qu(A)f(\qu)=\phi\not=0,$ by the continuity of $g$, there exists a neighborhood $V$ of $\qu$ in $\quat$ on which $g$ does not vanish. Therefore $\qu\in S(A)$. The claim is proved.\\
 	Claim: $\rka(A)\cap\sigma_S(A)\subseteq S(A)\cup S(A^\dagger).$\\
 	From proposition \ref{su1} we have $\sus(A)=\apo(A^\dagger)$. From part (a) and equation \ref{Ke3} we get
 	\begin{equation}\label{Ke4}
 	\rka(A)\cap\sus(A)=\rka(A^\dagger)\cap\apo(A^\dagger)\subseteq S(A^\dagger).
 	\end{equation}
 	Also from equation \ref{sue1}, we have $\sigma_S(A)=\sigma_{pS}(A)\cup\sus(A)$. Therefore,
 	\begin{eqnarray*}
 		\rka(A)\cap\sigma_S(A)&=&\rka(A)\cap(\sigma_{pS}(A)\cup\sus(A))\\
 		&=&(\rka(A)\cap\sigma_{pS}(A))\cup(\rka(A)\cap\sus(A)\\
 		&\subseteq&(\rka(A)\cap\apo(A))\cup(\rka(A)\cap\sus(A))\quad\text{by proposition \ref{P3}}\\
 			&=&S(A)\cup S(A^\dagger)\quad\text{by equations \ref{Ke3} and \ref{Ke4}}
 			\end{eqnarray*}
(c)~~From equations \ref{Ke3} and \ref{Ke4}, we also have
$$\rka(A)\cap\apo(A)\cap\sus(A)\subseteq S(A)\cap S(A^\dagger),$$
which means
$$(\apo(A)\cap\sus(A))\setminus(S(A)\cap S(A^\dagger)\subseteq\ska(A).$$
Thus, from equations \ref{Ke1} and \ref{Ke2}, we get (c).\\
(d)~~From equation \ref{sue1} and proposition \ref{P3}, we have
\begin{equation}\label{Ke5}
\sigma_S(A)=\sigma_{pS}(A)\cup\sus(A)\subseteq\apo(A)\cup\sus(A).
\end{equation}
Hence, from proposition \ref{KP0} and equation \ref{Ke5}, we get
\begin{equation}\label{Ke6}
\sigma_S(A)\setminus\apo(A)\subseteq\sus(A)\quad\text{and}\quad\sigma_S(A)\setminus\apo(A)\subseteq\rka(A).
\end{equation}
From equations \ref{Ke6} and \ref{Ke4} we also get
\begin{equation}\label{Ke7}
\sigma_S(A)\setminus\apo(A)\subseteq\rka(A)\cap\sus(A)\subseteq S(A^\dagger).
\end{equation}
Also from equation \ref{Ke5} we obtain $\sigma_S(A)\setminus\sus(A)\subseteq\apo(A)$. Further, proposition \ref{CP1} and equation \ref{sue1} yield $\sigma_S(A)\subseteq \apo(A)\cup\sus(A)$.
 Hence, from proposition \ref{su1} part (b), proposition \ref{KP0} and part (a) of this proposition, we get $\sigma_S(A)\setminus\sus(A)\subseteq\rka(A)$. Therefore, we have
\begin{equation}\label{Ke8}
\sigma_S(A)\setminus\sus(A)\subseteq\rka(A)\cap\apo(A)\subseteq S(A).
\end{equation}
From  equations \ref{Ke7} and \ref{Ke8} we get the inclusion 
$$(\apo(A)\setminus S(A))\cup (\sus(A)\setminus S(A^\dagger))\subseteq(\sus(A)\cap\apo(A))\setminus(S(A)\cap S(A^\dagger).$$
The opposite inclusion is trivial, and hence we have (d).
	 \end{proof}
%%%%%%%%%%%%%%%
The sandwich formula of proposition \ref{3.1.6} yields a precise description of the Kato S-spectrum when one of the sets $S(A)$ or $S(A^\dagger)$ is empty, which is another way of saying that $A$ or $A^\dagger$ has SVEP. We present it in the following corollary.
\begin{corollary}\label{KC1}
Let $A\in\B(\vr)$.
\begin{enumerate}
	\item [(a)] If $A$ has SVEP, then $\ska(A)=\apo(A)$.
	\item[(b)] If $A^\dagger$ has SVEP, then $\ska(A)=\sus(A)$,
	\item[(c)] If $A$ and $A^\dagger$ have SVEP, then $\ska(A)=\sigma_S(A)$.	
\end{enumerate}
\end{corollary}
\begin{proof}
(a)~Suppose $A$ has SVEP, then by proposition \ref{ResP}, $S(A)=\emptyset$. By equation \ref{Ke3} and proposition \ref{KP0} we have
$$\apo(A)\setminus S(A)\subseteq\ska(A)\subseteq\apo(A).$$
Therefore $\apo(A)=\ska(A)$.\\
(b)~~Similarly, if $A^\dagger$ has SVEP then $S(A^\dagger)=\emptyset$. But by proposition \ref{3.1.6} part(a), proposition \ref{su1} part (b) and equation \ref{Ke4}, we have
$$\sus(A)\setminus S(A^\dagger)\subseteq\ska(A)\subseteq\sus(A),$$
and hence $\ska(A)=\sus(A)$.\\
(c)~If $A$ and $A^\dagger$ have SVEP, then by proposition \ref{LP1} part (d) and by the above two parts we get $\ska(A)=\sigma_S(A)$.
\end{proof}
%%%%%%%%%%%%%%%%%%%
\begin{remark}\label{KR2}
Let $A\in\B(\vr)$ is a non-invertible isometry 	 and $A$ has SVEP. Then by remark \ref{SR1} we have $\apo(A)\subseteq\partial B_\quat(0,1)$ and  $\sigma_S(A)=\nabla_\quat(0,1)$. Further, from proposition \ref{T1}, we have $\partial\sigma_S(A)\subseteq\apo(A)$. Therefore, from corollary \ref{KC1}, $\ska(A)=\apo(A)=\{\qu\in\quat~~|~~|\qu|=1\}=\partial B_\quat(0,1)$. Also, in this case, $\rka(A)\cap\sigma_S(A)=(\quat\setminus\{\qu\in\quat~~|~~|\qu|=1\})\cap\nabla_\quat(0,1)=\{\qu\in\quat~~|~~|\qu|< 1\}=B_\quat(0,1).$ 
\end{remark}

%%%%%%%%%%%%%%%%
\begin{theorem}\label{1.46AII}
	Let $A\in\B(\vr)$, $M$ and $N$ be two $A$-invariant closed subspaces of $\vr$ and $\vr=M\oplus N$. Then $A$ is semi-regular if and only if $A|_M$ and $A|_N$ are semi-regular. Consequently, $\ska(A)=\ska(A|_M)\cup\ska(A|_N)$.
\end{theorem}
\begin{proof}
The equality $\kr(A|_M)=M\cap\kr(A)$ is trivial. Let us show that $A(M)=M\cap A(\vr)$. Since $M$ is $A$-invariant, trivially $A(M)\subseteq M\cap A(\vr)$. Conversely, if $\psi\in M\cap A(\vr)$, then $\psi\in M$ and $\psi=A(\phi)$ for some $\phi\in\vr$. Write $\phi=\phi_1+\phi_2$ with $\phi_1\in M$ and $\phi_2\in N$. Then $\psi=A(\phi)=A(\phi_1)+A(\phi_2)$, and since $A(\phi_1)\in M$ we have $A(\phi_2)=\psi-A(\phi_1)\in M\cap N=\{0\}$. Therefore $\psi=A(\phi_1)\in A(M)$. Thus $A(M)=M\cap A(\vr)$. By induction we have $(A|_M)^n(M)=A^n(M)=M\cap A^n(\vr)$ for all $n\in\N$. Assume that $A$ is semi-regular. Then
$$\kr(A|_M)=M\cap\kr(A)\subseteq M\cap A^n(\vr)=(A|_M)^n(\vr)\quad\text{for all}~~n\in\N.$$
Moreover $(A|_M)(M)=M\cap A(\vr)$ is closed, and hence $A|_M$ is semi-regular. In the same way we obtain that $A|_N$ is semi-regular. Conversely, if $A|_M$ and $A|_N$ are semi-regular, then $A(\vr)=A(M)\oplus A(N)$ is closed and
$$\kr(A)=\kr(A|_M)\oplus\kr(A|_N)\subseteq A^n(M)\oplus A^n(N)=A^n(\vr)$$
for all $n\in\N$. Therefore $A$ is semi-regular. As a consequence $R_\qu(A)$ is semi-regular if and only if $R_\qu(A)|_M$ and $R_\qu(A)|_N$ are semi-regular. Therefore $\ska(A)=\ska(A|_M)\cup\ska(A|_N)$.
\end{proof}

%%%%%%%%%%%%%%%%%%%%%%%%%%%%%%%%%%%%%%%%%%%%%%%%%%%%
\section{Generalized Kato decomposition}
In this section we introduce an important property of decomposition for
bounded operators which involves the concept of semi-regularity and nilpotent nature. We define the generalized Kato decomposition in the quaternionic setting following its complex counterpart. For the complex theory we refer the reader to \cite{Ai,Ben}.
\begin{definition}\label{GD1}
An operator $A\in\B(\vr)$ is said to be nilpotent of order $d\in\N$ if $A^d=0$ while $A^{d-1}\not=0$. It is said to be quasi-nilpotent if $\displaystyle\lim_{n\rightarrow\infty}\|A^n\|^{1/n}=0$.
\end{definition}
\begin{proposition}\label{GP1}
	Let $A\in\B(\vr)$. If $A$ is quasi-nilpotent then $\sigma_S(A)=\{0\}$.
\end{proposition}
\begin{proof}
	The S-spectral radius of $A\in\B(\vr)$ is defined as $r_S(A)=\sup\{|\qu|~|~\qu\in\sigma_S(A)\}$, see \cite{Jo} page 90. By theorem 4.2.3 of \cite{Jo}, $\displaystyle r_S(A)=\lim_{n\rightarrow\infty}\|A^n\|^{1/n}.$ Therefore, if $A$ is quasi-nilpotent, then $r_S(A)=0$, and hence $\sigma_S(A)=\{0\}$.
\end{proof}
\begin{definition}\label{GD2}
	An operator $A\in\B(\vr)$ is said to admit a generalized Kato decomposition, abbreviated as GKD, if there exists a pair of $A$-invariant closed right linear subspaces $(M,N)$ such that $\vr=M\oplus N$, the restrictions $A|_M$ is semi-regular and $A|_N$ is quasi-nilpotent.
\end{definition}
For example, every semi-regular operator has a GKD $M=\vr$ and $N=\{0\}$. Every quasi-nilpotent operator has a GKD, $M=\{0\}$ and $N=\vr$.
\begin{definition}\label{GD3}
	In definition \ref{GD2}, if $A|_N$ is nilpotent then there exists $d\in\N$ such that $(A|_N)^d=0$. In this case $A$ is said to be Kato type of order $d$. In general any such operator is said to be of Kato type.
\end{definition}
\begin{definition}\label{GD4}
	An operator $A\in\B(\vr)$ is said to be essentially semi-regular if it admits a GKD $(M,N)$ such that $N$ is finite dimensional.
\end{definition}
\begin{proposition}\label{GP2}
	Every quasi-nilpotent operator on a finite dimensional $\vr$ is nilpotent.
\end{proposition}
\begin{proof}
	Suppose $\vr$ is finite dimensional, $\dim(\vr)=n<\infty$ and $A\in\B(\vr)$ is quasi-nilpotent. Then $\sigma_S(A)=\{0\}$, also $A$ is an $n\times n$ quaternionic invertible matrix. Since $\sigma_S(A)=\{0\}$, by the Jordan canonical form, $A$ is similar to a matrix whose only non-zero entries are on the super-diagonal (see \cite{Al} section 4.3). In turn this is equivalent to $A^k=0$ for some $k\in\N$. 
\end{proof}
\begin{remark}\label{GR1}
	From proposition \ref{GP2}, if $A\in\B(\vr)$ is essentially semi-regular then $A|_N$ is nilpotent. Thus we have the following implications:\\
	$A$ is semi-regular~$\Rightarrow$~ $A$ is essentially semi-regular~$\Rightarrow$ $A$ is of Kato type.
\end{remark}
%%%%%%%%%%%%%%%%%%%%%%%%%%%
\begin{theorem}\label{1.41}
	Suppose that $(M,N)$ is a GKD for $A\in\B(\vr)$. Then we have
	\begin{enumerate}
		\item [(a)] $K(A)=K(A|_M)$ and $K(A)$ is closed.
		\item[(b)] $\kr(A|_M)=\kr(A)\cap M=K(A)\cap\kr(A)$
	\end{enumerate}
\end{theorem}
\begin{proof}
	A proof follows its complex counterpart. For a complex proof see theorem 1.41 in \cite{Ai}.
\end{proof}
\begin{theorem}\label{1.42}
	Let $A\in\B(\vr)$, and assume that $A$ is of Kato type of order $d\in\N$ with a GKD $(M,N)$. Then,
	\begin{enumerate}
		\item [(a)] $K(A)=A^\infty(\vr)$;
		\item[(b)] $\kr(A|_M)=\kr(A)\cap A^\infty(\vr)=\kr(A)\cap A^n(\vr)$ for all $d\leq n\in\N$;
		\item[(c)] $A(\vr)+\kr(A^n)=A(M)\oplus N$ for all $d\leq n\in\N$. Moreover $A(\vr)+\kr(A^n)$ is closed in $\vr$.
	\end{enumerate}
\end{theorem}
\begin{proof}
	A proof follows its complex counterpart. For a complex proof see theorem 1.42 in \cite{Ai}.
\end{proof}
%%%%%%%%%%%%%
\begin{theorem}\label{1.44}
	Let $A\in\B(\vr)$ be of Kato type. Then there exists an open quaternion sphere $B_\quat(0,\epsilon)\subseteq\quat$ for which $R_\qu(A)$ is semi-regular for all $\qu\in B_\quat(0,\epsilon)\setminus\{0\}$.
\end{theorem}
\begin{proof}
Let $(M,N)$ be a GKD for $A$ such that $A|_N$ is nilpotent.\\
Claim: $R_\qu(A)(\vr)$ is closed for all $\qu\in\quat$ for which $\gamma(A|_M)>\beta(\qu)$.\\
Since $A|_N$ is nilpotent, $R_\qu(A|_N)$ is bijective for all $\qu\not=0$. Thus, $N=R_\qu(A|_N)(N)$ for all $\qu\not=0$, and therefore
$$R_\qu(A)(\vr)=R_\qu(A)(M)\oplus R_\qu(A)(N)=R_\qu(A)(M)\oplus N,~~\text{for all}~~\qu\not=0.$$
By assumption $A|_M$ is semi-regular, so by theorem \ref{KT1} $R_\qu(A|_M)$ is semi-regular for all $\qu$ for which $\gamma(A|_M)>\beta(\qu)$. So $R_\qu(A|_M)$ is a closed subspace of $M$ for all $\qu$ for which $\gamma(A|_M)>\beta(\qu)$.
Consider the Hilbert space $M\times N$ provided with the canonical norm $$\|(\phi,\psi)\|=\|\phi\|+\|\psi\|,~~~\phi\in M,~~\psi\in N$$
and let $\Psi:M\times N\longrightarrow M\oplus N=\vr$ denote the topological isomorphism defined by $\Psi(\phi,\psi)=\phi+\psi$ for every $\phi\in M$ and $\psi\in N$. Then, for all $\qu$ for which $\gamma(A|_M)>\beta(\qu)$, since the set $R_\qu(A)(M)\times N$ is closed in $M\times N$, the set
$$\Psi(R_\qu(A)(M)\times N)=R_\qu(A)(M)\oplus N=R_\qu(A)(\vr)$$
is closed.\\
{\em Claim}: There is an open ball $B_\quat(0,\epsilon)$ such that $N^\infty(R_\qu(A))\subseteq R_\qu(A)^\infty(\vr)$ for all $\qu\in B_\quat(0,\epsilon)\setminus\{0\}$.\\
Since $A$ is of Kato type, by theorem \ref{1.41} and theorem \ref{1.42}, the hyper-range is closed and coincides with $K(A)$, consequently by theorem \ref{AT1}, $A(A^\infty(\vr))=A^\infty(\vr)$. Let $A_0=A|_{A^\infty(\vr)}$. The operator $A_0$ is onto and hence, by part (b) of proposition \ref{KP1}, $ R_\qu(A_0)$ is onto for all $\qu$ for which $\gamma(A_0)>\beta(\qu)$. Therefore $R_\qu(A)(A^\infty(\vr))=A^\infty(\vr)$  for all $\qu$ for which $\gamma(A_0)>\beta(\qu)$. Then, by theorem \ref{AT2}, $A^\infty(\vr)$ is closed, and we infer that
$A^\infty(\vr)\subseteq K(R_\qu(A))\subseteq R_\qu(A)^\infty(\vr)$  for all $\qu$ for which $\gamma(A_0)>\beta(\qu)$. By theorem \ref{1.3} part (b), we conclude that
\begin{equation}\label{GE1}
N^\infty(R_\qu(A))\subseteq (A^2)^\infty(\vr)\subseteq A^\infty(\vr)\subseteq R_\qu(A)^\infty(\vr)
\end{equation}
 for all $\qu\not=0$ for which $\gamma(A_0)>\beta(\qu)$. The inclusion in equation \ref{GE1} together with $R_\qu(A)(\vr)$ being closed for all $\qu$ for which $\gamma(A|_M)>\beta(\qu)$, then imply the semi-regularity of $R_\qu(A)$ for $0<|\qu|\leq\beta(\qu)<\epsilon$, where $\epsilon=\min\{\gamma(A_0), \gamma(A|_M)\}>0.$
\end{proof}
%%%%%%%%%%%%%%%%%%%
\begin{definition}\label{GD5}
	Let $A\in\B(\vr)$, then the generalized Kato S-spectrum is defined as
	$$\sgka(A)=\{\qu\in\quat~~|~~R_\qu(A)~~\text{is not of Kato type}\}$$
	and the generalized Kato S-resolvent is $\rgka(A)=\quat\setminus\sgka(A)$. The essentially S-semi-regular spectrum and its resolvent are defined, respectively, by
	$$\ses(A)=\{\qu\in\quat~~|~~R_\qu(A)~~\text{is not essentially semi-regular}\}$$
	and $\res(A)=\quat\setminus\ses(A)$.
\end{definition}
From remark \ref{GR1}, clearly, for $A\in\B(\vr)$, we have $$\sgka(A)\subseteq\ses(A)\subseteq\ska(A)\subseteq\sigma_S(A).$$
%%%%%%%%%%%%%%%%%%%%
\begin{corollary}\label{1.45}
	If $A\in\B(\vr)$, then $\sgka(A)$ and $\ses(A)$ are compact subsets of $\sigma_S(A)$. Moreover, $\ska(A)\setminus\sgka(A)$ and $\ses(A)\setminus\sgka(A)$ consists of at most countably many isolated points.
\end{corollary}
\begin{proof}
	From theorem \ref{1.44}, clearly $\rgka(A)=\quat\setminus\sgka(A)$ and  $\res(A)=\quat\setminus\ses(A)$ are open subsets of $\quat$, and hence $\sgka(A)$ and $\ses(A)$ are closed subsets of the compact set $\sigma_S(A)$. Therefore, $\sgka(A)$ and $\ses(A)$ are compact subsets of $\sigma_S(A)$. If $\qu_0\in\ses(A)\setminus\sgka(A)$ then $R_\qu(A)$ is semi-regular as $\qu$ belongs to a suitable punctured ball centered at $\qu_0$. Hence, $\ses(A)\setminus\sgka(A)$ consists of at most countably many isolated points, and the same argument is true for  $\ska(A)\setminus\sgka(A)$.
\end{proof}
	
%%%%%%%%%%%%%%%%%%%%%%
\section{conclusion}
We have studied the surjectivity S-spectrum, Kato S-spectrum, generalized Kato spectrum, essentially semi-regular S-spectrum and  approximate S-point spectrum of a bounded right linear operator on a right quaternionic Hilbert space $\vr$ without introducing a left multiplication in $\vr$. We have also established various connections between these spectra. In particular, we have proved that the Kato S-spectrum is a non-empty compact subset of the S-spectrum.\\

We have also introduced and studied local S-spectrum $\sigma_A(\phi)$ at a point $\phi\in\vr$ and the local S-spectral subspace $X_A(F)$ of a bounded right linear operator $A$ associated with a set $F$ to certain extent. In the complex theory, the local spectrum $\sigma_A(\phi)$ and local spectral set $X_A(F)$ play an important part, as theory itself, in establishing several important results regarding the Kato, generalized Kato and many other parts of the spectrum. In particular, the equality, for a vector $\phi$ in the complex Hilbert space $\HI$ and $\lambda\in\C$,
\begin{equation}\label{CE}
\sigma_A(\phi)=\sigma_A(f(\lambda)),
\end{equation}
where $f:U\longrightarrow\HI$ is an analytic function defined in an open neighborhood $U$ of $\lambda$ for which $(A-\lambda\mathbb{I}_{\HI})f(\mu)=\phi$ for all $\mu\in U$, see theorem 2.2 in \cite{Ai} or theorem 1.2.14 in \cite{La}. Unfortunately, under the current set up of the manuscript, we have experienced difficulty in establishing an identity similar to equation \ref{CE}. This fact have affected our ability in establishing several results valid in the complex case to quaternions. In particular, we have shown that the generalized Kato S-spectrum is a compact subset of the S-spectrum, however, we were unable to show that the non-isolated points of $\partial\sigma_S(A)$ belongs to $\sgka(A)$ which is the case in the complex setting.

%%%%%%%%%%%%%%%%%%%%%%%%%%%%%%%%%%%%%%%%%%%%%%%%%
\section{Acknowledgments}
K. Thirulogasanthar would like to thank the FRQNT, Fonds de la Recherche  Nature et  Technologies (Quebec, Canada) for partial financial support under the grant number 2017-CO-201915. Part of this work was done while he was visiting the University of Jaffna to which he expresses his thanks for the hospitality.

\end{document}